\documentclass{amsart}
\usepackage{amsfonts}
\usepackage{amsmath,amssymb}
\usepackage{amsthm}
\usepackage{amscd}
\usepackage{graphics}
\usepackage{graphicx}
{
\theoremstyle{remark}
\newtheorem{Def}{{\rm Definition}}
\newtheorem{Ex}{{\rm Example}}

}
\newtheorem{Cor}{Corollary}
\newtheorem{Prop}{Proposition}
\newtheorem{Thm}{Theorem}

\begin{document}
\title[Homological information of fold maps given by surgerying]{Constructing fold maps by surgery operations and homological information of their Reeb spaces}
\author{Naoki Kitazawa}
\address{Industry of Mathematics for Industry, Kyushu University, 744 Motooka, Nishi-ku Fukuoka 819-0395, Japan}
\email{n-kitazawa.imi@kyushu-u.ac.jp}
\maketitle
\begin{abstract}
In this paper, as a fundamental study on the theory of {\it Morse} functions and their higher dimensional versions or {\it fold maps} and applications
 to geometric theory of manifolds, which were started in 1950s by differential topologists such as Thom and Whitney and have been studied actively, we study algebraic and differential
 topological properties of certain fold maps and their source manifolds. More precisely, we investigate fold maps obtained by surgery operations to fundamental
   fold maps and especially, homology groups of {\it  Reeb spaces}, which are defined as the space of all connected components of inverse images, often inheriting important invariants of manifolds such as homology groups and fundamental and important tools in studying manifolds.  
   
    Studies of this paper are especially motivated by the stream of studies of fold maps satisfying good (differential) topological properties such as {\it special generic} maps, which
    were defined in 1970s and studied since 1990s by Saeki and Sakuma, and {\it round} fold maps, which were introduced by the author in 2012--2014, and their source manifolds. Moreover, constructions
     of generic maps by fundamental surgeries to investigate manifolds by using generic maps, studied by Kobayashi and Saeki etc., also have motivated the present study.   
         
\end{abstract}


\maketitle
\section{Introduction and fundamental notation and terminologies}
\label{sec:1}
\subsection{Historical backgrounds}
Morse functions and {\it fold maps}, which are regarded as higher dimensional versions of Morse functions, play important roles
 in studying smooth
 manifolds by using generic maps since related studies were started by Thom (\cite{thom}) and Whitney (\cite{whitney}).

  Let $m$ and $n$ be integers satisfying $m \geq n \geq 1$. A {\it fold map} from
 an $m$-dimensional closed smooth manifold into an $n$-dimensional smooth
 manifold without boundary is defined as a smooth map
 such that each singular point is of the form
$$(x_1, \cdots, x_m) \mapsto (x_1,\cdots,x_{n-1},\sum_{k=n}^{m-i}{x_k}^2-\sum_{k=m-i+1}^{m}{x_k}^2)$$
 for an integer $0 \leq i \leq \frac{m-n+1}{2}$ (the integer $i$ is uniquely
 determined and it is called the {\it index} of the singular point). We easily have the following.
\begin{Prop}
\label{prop:1}
\begin{enumerate}
\item The set consisting of all singular points {\rm (}the {\it singular set}{\rm )} of the map is a smooth closed submanifold of dimension $n-1$ of the source manifold. 
\item The restriction map to the singular set is a smooth immersion of codimension $1$.
\end{enumerate}
\end{Prop} 
We also note that if the restriction map to the singular set is an immersion with normal crossings, then it is {\it stable} and stable fold maps satisfy this property : {\it stable} maps are generic maps and important in the theory of global singularity; see \cite{golubitskyguillemin} for example. To know fundamental properties of fold maps, see also \cite{golubitskyguillemin} and \cite{saeki} for example.

 As a branch of the theory of fold maps, stable maps and generic maps and its applications to studies of smooth
 manifolds, algebraic and differential topological properties of fold maps of several
 classes such as {\it special generic} maps, {\it simple} fold maps and {\it round} fold maps and manifolds admitting such
 maps have been studied
 since 1990s.

 A {\it special generic} map
 is defined as a fold map such that the indices of singular points are always $0$. Special generic maps were first introduced and studied in \cite{furuyaporto} and differential topological theory of special generic maps has developed owing to studies of Saeki and Sakuma especially in1990s. A Morse function on a homotopy
 sphere with just
 two singular points is regarded as a simplest special generic map; every smooth
 homotopy sphere of dimension $k \neq 4$ and the $4$-dimensional standard sphere $S^4$ admits such
 a function and conversely, a
 manifold admitting such a function is a homotopy sphere (see \cite{milnor} and \cite{milnor2} and see also \cite{reeb}). Moreover, we easily
 obtain a special generic map from any standard sphere of dimension $k_1 \geq 2$ into the $k_2$-dimensional Euclidean space ${\mathbb{R}}^{k_2}$ by the canonical
 projection of the unit sphere under the assumption that $k_1 \geq k_2 \geq 1$ holds (FIGURE \ref{fig:1}). 

\begin{figure}
\includegraphics[width=40mm]{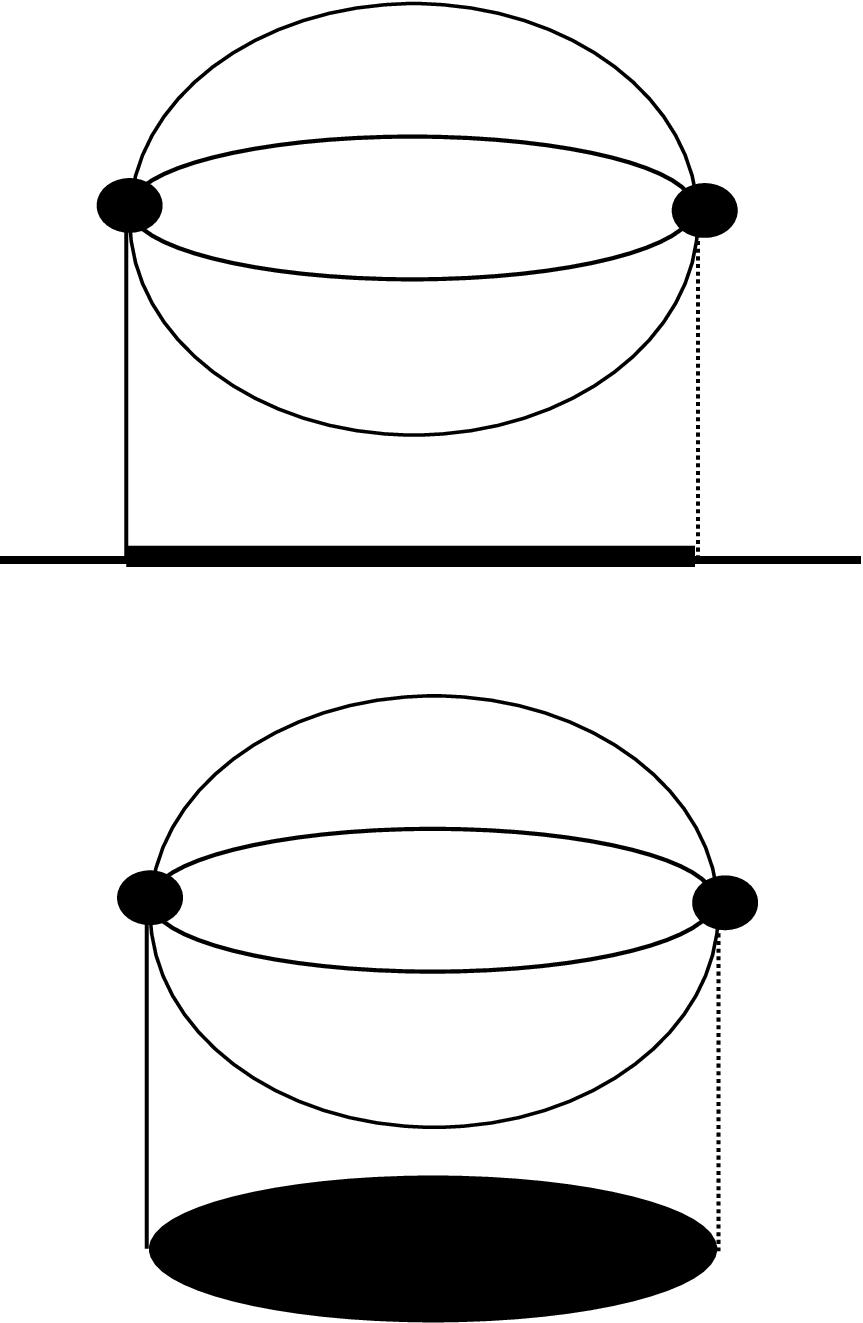}
\caption{The canonical projections of the unit spheres for the cases $k_2=1$ and $k_2 \geq 2$.}
\label{fig:1}
\end{figure}

On the other hand, it was
 shown that a homotopy sphere of dimension $k_1$ admitting a special generic map
 into ${\mathbb{R}}^{k_2}$ is diffeomorphic to the standard sphere $S^{k_1}$ under the assumption that $1 \leq k_1-k_2 \leq 3$ holds (see \cite{saeki2} and \cite{saeki3}). In addition, in \cite{saeki2}, \cite{saeki3}
 and \cite{saekisakuma}, the diffeomorphism types of manifolds admitting special generic maps into ${\mathbb{R}}^2$ and ${\mathbb{R}}^3$ are completely
 or partially classified. A Morse function and its singular points tells us homology groups and some information
 on homotopy of the source manifold and a special generic map and its singular points often tells
 us more; the homeomorphism and diffeomorphism type of the source manifold.

  Such interesting properties of special generic maps imply that considering appropriate classes
   of (stable)
 fold maps and studying algebraic and differential topological properties of such maps and manifolds admitting them
 is essential in the theory of fold maps. As generalizations, {\it simple} fold maps and fold maps such that connected components of inverse images containing no singular points are spheres have been also studied in \cite{saeki} and \cite{sakuma} for example. Motivated by this stream, {\it round} fold maps were
 introduced in 2012--2014 by the author in \cite{kitazawa2} and \cite{kitazawa3}.

\begin{Def}[\cite{kitazawa2} and \cite{kitazawa3}]
\label{def:1}
A {\it round} fold map is a fold map into an Euclidean space of dimension larger than $1$ satisfying the following three.
\begin{enumerate}
\item The singular set is a disjoint union of standard spheres.
\item The restriction map to the singular set is an embedding.
\item The set consisting of all the singular values (the {\it singular value set}) of the map is a disjoint union
 of spheres embedded concentrically. 
\end{enumerate}
\end{Def} 

See also FIGURE \ref{fig:2}, representing the singular value set of a round fold map into the plane.

\begin{figure}
\includegraphics[width=40mm]{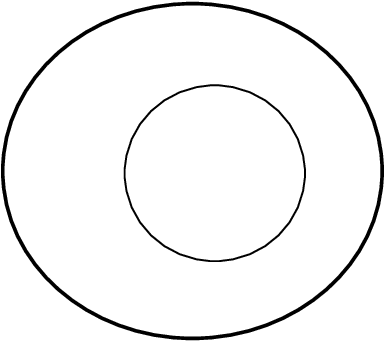}
\caption{The singular value set of a round fold map into the plane.}
\label{fig:2}
\end{figure}

 Some special generic maps are round fold maps. The standard sphere of dimension $m>1$
 admits such a map into ${\mathbb{R}}^n$ under the assumption that $m \geq n \geq 2$ holds; we can construct the map by the natural
 projection before and the resulting singular set is the ($n-1$)-dimensional standard sphere as shown also in FIGURE \ref{fig:1}. Furthermore, in \cite{saeki}, Saeki constructed such a map into the plane on every homotopy sphere whose
 dimension is larger than $1$ and not $4$.

 Homology and homotopy groups and more precisely, the
 homeomorphism and diffeomorphism types of manifolds admitting round fold maps were studied in \cite{kitazawa}, \cite{kitazawa2}, \cite{kitazawa3} and \cite{kitazawa4} under appropriate
 conditions by the author. For example, homology and homotopy groups of manifolds admitting round fold maps such that the inverse images of regular values are
 disjoint unions of spheres were studied in \cite{kitazawa2} and \cite{kitazawa3}. In addition, the following has been shown through explicit constructions of
 fold maps in \cite{kitazawa}, \cite{kitazawa2} or \cite{kitazawa4}.

\begin{Prop}[ \cite{kitazawa}, \cite{kitazawa2} or \cite{kitazawa4}]
\label{prop:2}

Let $m$ and $n$ be integers satisfying $m>n \geq 2$. 
\begin{enumerate}
\item
The total
 space of a smooth bundle over $S^n$ whose fiber is diffeomorphic
 to a homootpy sphere $\Sigma$ obtained by gluing two copies of a standard closed disc on the boundaries by a diffeomorphism admits a round fold map into ${\mathbb{R}}^n$ satisfying
 the following conditions.

\begin{enumerate}
\item The singular set consists of $2$ connected components.
\item The inverse image of a regular value in
 the connected component of the regular value set, which
 is an open disc, is a disjoint union of two copies of the homotopy sphere $\Sigma$
  and these homotopy spheres are both fibers of the bundle.
\end{enumerate}
\item A smooth manifold
 represented as a connected sum of $l$ smooth manifolds regarded as
 the total spaces of smooth
 bundles over $S^n$ whose fibers are diffeomorphic
 to $S^{m-n}$ admits a round fold map into ${\mathbb{R}}^n$ satisfying
 the following conditions.
\begin{enumerate}
\item The inverse image of each regular value is a disjoint union of standard spheres
 and regarded as a fiber of a smooth bundle above.
\item The number of connected components of the inverse image
 of a regular value in the connected component of the regular value set, which
 is an open disc, is $l$.
\end{enumerate}
\end{enumerate}
\end{Prop}

\begin{figure}
\includegraphics[width=80mm]{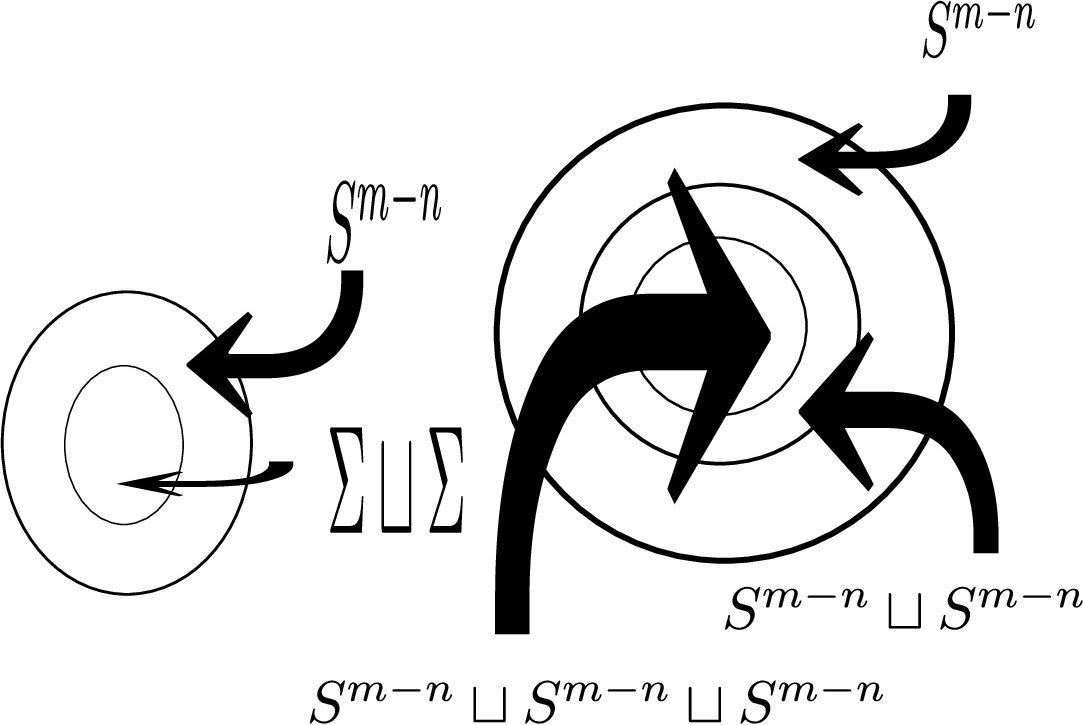}
\caption{Round fold maps into ${\mathbb{R}}^n$ ($n \geq 2$) of Proposition \ref{prop:2} (each manifold represents the corresponding inverse image of each regular value).}
\label{fig:3}
\end{figure}

  Note that explicit constructions are fundamental problems in the theory of fold maps, that succeeding in constructions help us to know precise structures of manifolds such as topological properties
 of submanifolds obtained as inverse images of regular values and that they will help us to study manifolds by the theory of fold
 maps easy to handle. However, constructions are often difficult except cases such as special generic maps and some round fold maps mentioned here.
  
\subsection{Main works of this paper and fundamental notation and terminologies}

 In this paper, we study algebraic
 and differential topological properties of fold
 maps which are not always round
 or special generic and manifolds admitting such maps. For example, we mainly study structures of fold maps obtained
  inductively by iterations of fundamental operations to algebraic and differential topologically simpler fold maps. In these studies, we introduce operations consisting of
 surgeries or operations of making singular value sets larger and exchanging inverse images of the maps locally and construct several new explicit more complex fold maps. Such operations were also introduced
  by Kobayashi and Saeki in \cite{kobayashi} and \cite{kobayashisaeki} for example and based on the
   ideas of the operations, we introduce these new operations. 

  Before introducing the operations, we review the {\it Reeb space} of a smooth
 map, which is defined as the space consisting of all connected components of
 inverse images of the map and is a fundamental tool in studying the manifold. In fact, the Reeb
 spaces often inherit fundamental invariants of source manifolds such as homology groups.  
 
  Next, we introduce a {\it {\rm (}normal{\rm )} bubbling operation} and as specific versions, an {\it M-bubbling} operation and an {\it S-bubbling operation} to a (stable) fold map (Definition \ref{def:3} and Definition \ref{def:5}). Through \cite{kobayashi} and \cite{kobayashi2}, Kobayashi introduced
 a {\it bubbling surgery} as a surgery operation to obtain explicit stable fold maps and he
 succeeded in constructing explicit (stable) fold maps and bubbling operations
 are regarded as extensions of bubbling surgeries. We apply bubbling operations
 to fold maps and we obtain new maps. We can easily perform such operations to stable fold maps (Example \ref{ex:1} (\ref{ex:1.1})) and as specific cases, for example, by the definition, if we perform an S-bubbling operation to a
  fold map such that (the restriction map to the singular set is an embedding and that) the inverse images of regular values are always disjoint unions
   of spheres as presented in Proposition \ref{prop:2}, then we can obtain a fold map satisfying a similar condition (Example \ref{ex:1} (\ref{ex:1.4})). We mainly study Reeb spaces
 of maps obtained by finite iterations of (normal) bubbling operations starting from a simpler fold map. As main theorems, we see that Euler numbers and homology groups of the resulting Reeb spaces or equivalently, the difference of these numbers and groups must satisfy
  some conditions and that they are often also flexible; for infinitely many integers, we obtain maps and Reeb spaces so that the Euler numbers are the given numbers and we see that there
 are many types of homology groups of the resulting Reeb spaces: for example, on the other hands, if we only consider round fold maps such that the inverse images of regular values are
 disjoint unions of spheres, then homology groups of Reeb spaces are not
  so flexible (this fact is also pointed out in Example \ref{ex:2} (\ref{ex:2.1}) for example.) 
 
 We also introduce isomorphisms between (co)homology and homotopy
 groups of manifolds admitting stable fold maps such that the restriction map to the singular set is an embedding or more generally, {\it simple fold maps}, which will appear in Example \ref{ex:1.4} with the definition, and that the inverse images of regular
 values are disjoint unions of spheres and those groups of their Reeb spaces, which
 were discussed in \cite{saekisuzuoka} and \cite{suzuoka} and later in \cite{kitazawa2} and \cite{kitazawa3} (Proposition \ref{prop:4}). As mentioned before, by an iteration of S-bubbling operations starting from a fold
  map satisfying the condition posed here, we always obtain a fold map satisfying this condition. 
 
 By combining obtained results on
 Euler numbers and (co)homology groups of Reeb spaces and these relations, for example, we can observe that (the changes of) Euler numbers and homology groups of manifolds admitting fold maps obtained by finite iterations of normal S-bubbling operations starting from a simpler fold map satisfying the condition here must satisfy some conditions and that (the changes of) the Euler numbers and the homology groups are often also flexible.

 In this paper, for
 a smooth map $c$, we define the {\it singular set} of $c$ as the set consisting of all singular points of $c$ as in the presentation of the fundamental
 properties of fold maps before or Proposition \ref{prop:1} and denote it by $S(c)$. For the smooth map $c$, we call $c(S(c))$ the {\it singular value set} of $c$ as in the presentation of
 the definition of a round fold map or Definition \ref{def:1} before. We call ${\mathbb{R}}^n-c(S(c))$ the {\it regular value set} of $c$.

  We also note on (homotopy) spheres. In this paper, an {\it almost-sphere} of dimension $k$ means a homotopy sphere
 given by gluing two $k$-dimensional standard
 closed discs together by a diffeomorphism between the boundaries.

  We often use terminologies on (fiber) bundles in this paper (see also \cite{milnorstasheff} and \cite{steenrod}). For a topological space $X$, an {\it $X$-bundle} is a bundle
 whose fiber is $X$. A bundle whose structure group is $G$
 is said to be a {\it trivial} bundle if it is equivalent to the product bundle
 as a bundle whose structure group is $G$. A {\it linear} bundle is
 a smooth bundle whose fiber is a standard disc or a standard sphere and whose structure group consists of
 linear transformations on the fiber.

  Throughout this paper, we assume that $M$ is a smooth, closed and connected manifold of dimension $m$, that
 $N$ is a smooth manifold of dimension $n$ without boundary, that $f:M \rightarrow N$ is a smooth map and that $m \geq n \geq 1$. Manifolds are of class $C^{\infty}$ and maps
 between manifolds are also of class $C^{\infty}$ and fold maps unless otherwise stated
 in the proceeding sections. In addition, the structure groups of bundles such that the fibers are (smooth) manifolds are assumed to be
 (subgroups of) diffeomorphism groups.


\section{Reeb spaces and (normal) S-bubbling operations}
\label{sec:2}
\subsection{Definitions and fundamental properties of Reeb spaces and normal S-bubbling operations.}
First, we review Reeb spaces of maps (see also \cite{reeb}).
\begin{Def}
\label{def:2}
 Let $X$ and $Y$ be topological spaces. For $p_1, p_2 \in X$ and for a map $c:X \rightarrow Y$, 
 we define as $p_1 {\sim}_c p_2$ if and only if $p_1$ and $p_2$ are in
 a same connected component of $c^{-1}(p)$ for some $p \in Y$. ${\sim}_{c}$ is an equivalence relation.

\ \ \ We denote the quotient space $X/{\sim}_c$ by $W_c$ and call $W_c$ the {\it Reeb space} of $c$.
\end{Def}

 We denote the induced quotient map from $X$ into $W_c$ by $q_c$. We can define $\bar{c}:W_c \rightarrow Y$ uniquely
 so that the relation $c=\bar{c} \circ q_c$ holds.

  For a (stable) fold map $c$, the Reeb space $W_c$ is regarded as a polyhedron. For example, for a
 Morse function, the Reeb space
 is a graph and for a
 special generic map,
 the Reeb space is regarded as a smooth manifold immersed into the target manifold (see section 2 of \cite{saeki2}).

\begin{Def}
\label{def:3}
For a fold map $f:M \rightarrow N$, let $P$ be a connected component of the regular
 value set ${\mathbb{R}}^n-f(S(f))$. Let $S$ be a connected and orientable closed submanifold of
 $P$ and $N(S)$, ${N(S)}_i$ and ${N(S)}_o$ be small closed tubular neighborhoods
 of $S$ in $P$ such that ${N(S)}_i \subset N(S) \subset {N(S)}_o$
 holds. Furthermore, we can naturally regard ${N(S)}_o$ as a linear bundle whose fiber is
 an ($m-n+1$)-dimensional disc of radius $1$ and ${N(S)}_i$ and $N(S)$ are subbundles of the bundle ${N(S)}$ whose
 fibers are ($m-n+1$)-dimensional discs of radii $\frac{1}{3}$ and $\frac{2}{3}$, respectively. Let
 $f^{-1}({N(S)}_o)$ have a connected component $Q$ such that $f {\mid}_{Q}$
 makes $Q$ a bundle over ${N(S)}_o$.

Let us assume that there exist an $m$-dimensional closed manifold $M^{\prime}$ and
 a fold map $f^{\prime}:M^{\prime} \rightarrow {\mathbb{R}}^n$
 satisfying the following.
\begin{enumerate}
\item $M-{\rm Int} Q$ is a compact submanifold (with non-empty boundary) of $M^{\prime}$ of dimension $m$.
\item $f {\mid}_{M-{\rm Int} Q}={f}^{\prime} {\mid}_{M-{\rm Int} Q}$ holds.
\item ${f}^{\prime}(S({f}^{\prime}))$ is the disjoint union of $f(S(f))$ and $\partial N(S)$.
\item $(M^{\prime}-(M-Q)) \bigcap f^{-1}({N(S)}_i)$ is empty or ${{f}^{\prime}} {\mid}_{(M^{\prime}-(M-Q)) \bigcap f^{-1}({N(S)}_i)}$ makes $(M^{\prime}-(M-Q)) \bigcap f^{-1}({N(S)}_i)$ a bundle over $N(S)$.
\end{enumerate}
These assumptions enable us to consider the procedure of constructing $f^{\prime}$ from $f$ and we
 call it a {\it normal bubbling operation} to $f$ and, ${\bar{f}}^{-1}(S) \bigcap q_f(Q)$, which is homeomorphic
 to $S$, the {\it generating manifold} of the normal bubbling operation. 

Furthermore, let us suppose additional conditions.

\begin{enumerate}
\item
 ${{f}^{\prime}} {\mid}_{(M^{\prime}-(M-Q)) \bigcap f^{-1}({N(S)}_i)}$ makes $(M^{\prime}-(M-Q)) \bigcap f^{-1}({N(S)}_i)$ the disjoint union of two bundles over $N(S)$, then the procedure is called a {\it normal M-bubbling operation} to $f$.
\item ${{f}^{\prime}} {\mid}_{(M^{\prime}-(M-Q)) \bigcap f^{-1}({N(S)}_i)}$ makes $(M^{\prime}-(M-Q)) \bigcap f^{-1}({N(S)}_i)$ the disjoint union of two bundles over $N(S)$ and the fiber of one of the bundles is an almost-sphere, then the procedure is called a {\it normal S-bubbling operation} to $f$.
\end{enumerate}
\end{Def}

FIGURE \ref{fig:4} shows explicit examples of bubbling operations and FIGURE \ref{fig:5} shows
 explicit local changes of Reeb spaces by M-bubbling operations.
\begin{figure}
\includegraphics[width=80mm]{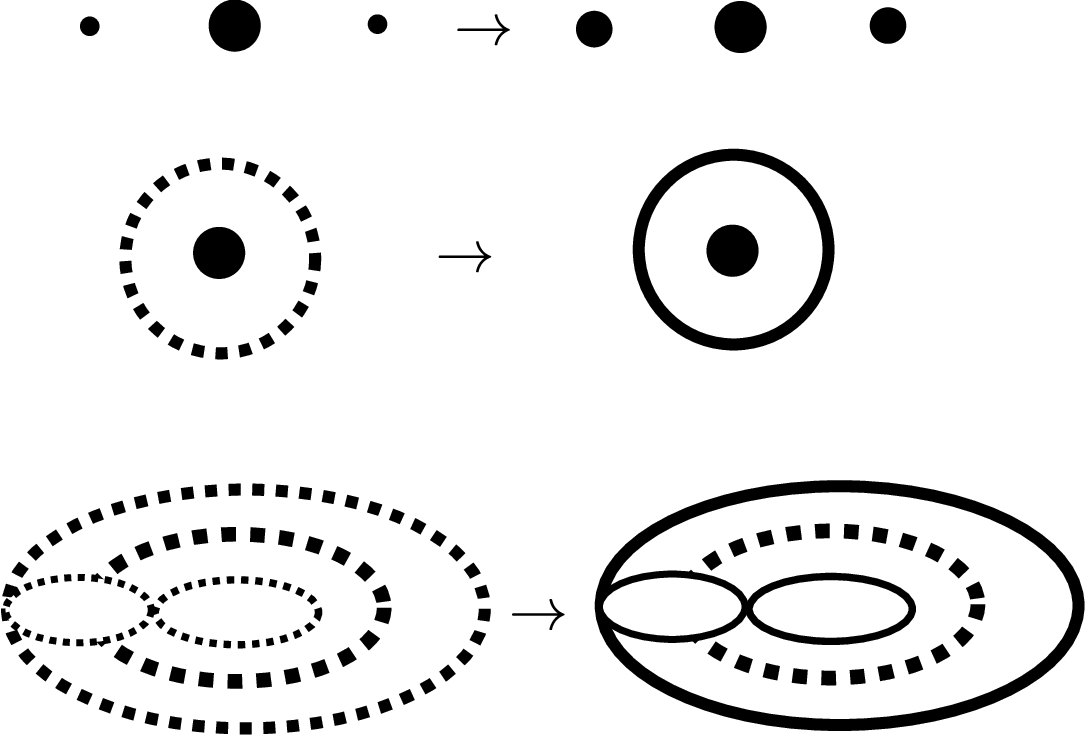}
\caption{Local changes of singular value sets by M-bubbling operations: $(n,k)=(1,0)$, $(n,k)=(2,0)$ and $(n,k)=(2,1)$ respectively.}
\label{fig:4}
\end{figure}

\begin{figure}
\includegraphics[width=80mm]{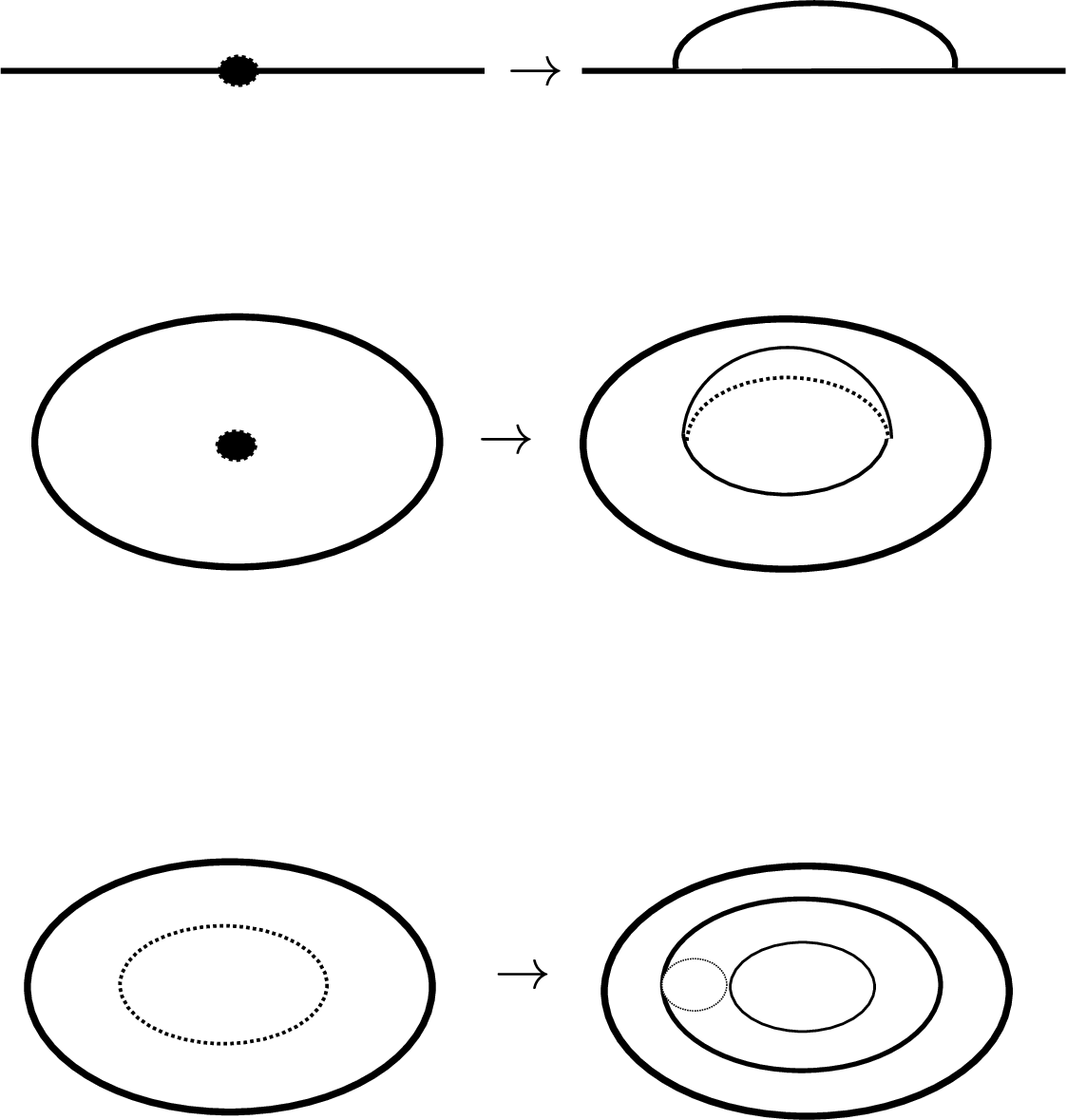}
\caption{Local changes of Reeb spaces by M-bubbling operations: $(n,k)=(1,0)$, $(n,k)=(2,0)$ and $(n,k)=(2,1)$ respectively.}
\label{fig:5}
\end{figure}
In the following example we present explicit and important facts on normal S-bubbling operations.

\begin{Ex}
\label{ex:1}
\begin{enumerate}
\item
\label{ex:1.1}
Let $f:M \rightarrow N$ be a fold map. Let $P$ be a connected component
 of the set $f(M)-f(S(f))$. Let $S$ be a connected and orientable closed submanifold of
 $P$ such that there exists a connected component $S^{\prime}$ of $f^{-1}(S)$ and that $f {\mid}_{S^{\prime}}:S^{\prime} \rightarrow S$
 makes $S^{\prime}$ a trivial bundle over $S$; for example, let $S$ be in the interior of an open ball in the interior
 of $P$. Let $f_{m,n,S}$ be a Morse function such that the following hold.

\begin{enumerate}
\item $f_{m,n,S}$ is a Morse function on a compact manifold one of connected components of whose boundary is the fiber
 $F$ of the bundle $S^{\prime}$ over $S$ and with just one singular point.
\item The inverse image of
 the maximal value is the component, diffeomorphic to $S$, and the inverse image of the minimal value is the disjoint
 union of connected components of the boundary except the previous one if the disjoint union is not empty and is the singular point of $f_{m,n,S}$ if the disjoint union is empty (or the index of the singular point is $0$).  
\end{enumerate} 

 Then, by a normal bubbling operation to $f$ such that the generating manifold is $S^{\prime}$, we
 can obtain a new fold map $f^{\prime}:M^{\prime} \rightarrow {\mathbb{R}}^n$ satisfying the following conditions
 where we abuse notation in Definition \ref{def:3}. We call this operation
 a {\it trivial} normal bubbling operation.   
\begin{enumerate}
\item $f {\mid}_{M-{\rm Int} Q}={f}^{\prime} {\mid}_{M-{\rm Int} Q}$.
\item ${{f}^{\prime}} {\mid}_{{{f}^{\prime}}^{-1}({N(S)}_i) \bigcap (M^{\prime}-(M-Q))}$ gives a disjoint
 union of two trivial bundles over ${N(S)}_i$ whose
 fibers are both almost-spheres. 
\item There exists a connected component of ${f^{\prime}}^{-1}(N(S)_o-{\rm Int}N(S)_i)$ such that the restriction map
 of ${f^{\prime}}$ to the component is regarded as the product of the Morse
 function $f_{m,n,S}$ and ${\rm id}_{\partial N(S)}$.
\end{enumerate}
Moreover, let the normal bundle or tubular neighborhood of $S$ be trivial. $N_o(S)$ is
 represented by $S \times D^{n-\dim S}$ and $S$ is regarded
  as $S \times \{0\} \subset S \times D^{n-\dim S}$. For example, let $S$ be the standard sphere embedded as an unknot in the interior of an open ball in the interior
 of $P$. Furthermore, let the
   restriction of $f^{\prime}$ to ${f^{\prime}}^{-1}(N(S)_o)$ is regarded as the product
    of a surjective map $f^{\prime} {\mid}_{{f^{\prime}}^{-1}(D^{n-\dim S})}:{f^{\prime}}^{-1}(D^{n-\dim S}) \rightarrow D^{n-\dim S}$ where $D^{n-\dim S}$ is a fiber of the trivial bundle $N(S)_o$ and ${\rm id}_{S}$. Then we call the previous operation a {\it strongly trivial} normal bubbling operation.  
\item
\label{ex:1.2}
In the previous example, if the fiber $F$ is orientable and $F_1$ and $F_2$ is a smooth, closed, connected and orientable
 so that by orienting $F_1$, $F_2$ and $F$ appropriately, $F$ is represented as a connected sum of $F_1$ and $F_2$, then, we can consider $f_{m,n.S}$ so that the boundary of
 its source manifold consists of three connected components and the boundary with the connected
 component $F$ removed is the disjoint union of $F_1$ and $F_2$.  In this case, the normal bubbling operation is an M-bubbling operation.
\item
\label{ex:1.3}
In (\ref{ex:1.1}), for any almost-sphere $\Sigma$ of dimension $m-n$, we can consider $f_{m,n.S}$ so that the boundary of
 its source manifold consists of three connected components and the boundary with the connected
 component $F$ removed is the disjoint union of $\Sigma$ and a manifold $F^{\prime}$ homeomorphic to $F$. 
 In this case, the operation is an S-bubbling operation.
\item
\label{ex:1.4}
If we perform a normal S-bubbling operation to a fold map such that the inverse images of regular values are always disjoint unions of almost-spheres (standard spheres), then we obtain a fold map satisfying this condition again.

A {\it simple} fold map is a fold map such that the map ${q_f} {\mid}_{S(f)}:S(f) \subset M \rightarrow W_f$ is
 injective (see \cite{saeki} and \cite{sakuma} for example). Any special generic map and fold map such that the map $f {\mid}_{S(f)}$ is an embedding are simple fold maps.
 If we perform a normal S-bubbling operation to a simple fold map (such that the map $f {\mid}_{S(f)}$ is an embedding), then so is the resulting fold map. 
\end{enumerate}
\end{Ex}

The following proposition is a fundamental and key tool in the present paper.

\begin{Prop}
\label{prop:3}
For a fold map $f:M \rightarrow N$, let $f^{\prime}:{M}^{\prime} \rightarrow N$ be a fold map obtained by a
 normal M-bubbling operation to $f$ and let $S$ be the generating manifold of the normal
 M-bubbling operation and of dimension $k<n$. Then for any PID $R$ and integer i, we have

$H_i(W_{{f}^{\prime}};R) \cong H_{i}(W_f;R) \oplus (H_{i-(n-k)}(S;R))$.
\end{Prop}

We prove Proposition \ref{prop:3} in the proof of Proposition \ref{prop:6} later; Proposition \ref{prop:3} follows immediately from Proposition \ref{prop:6}.

We have the following as a corollary.

\begin{Cor}
\label{cor:1}
In the situation of Proposition \ref{prop:3}, if $H_{i-n+k}(S;R)$ is free and $H_i(W_f;R)$ is free, then $H_i(W_{{f}^{\prime}};R)$ is also free.
\end{Cor}

In this paper, for a topological space $X$, we denote the Euler number of $X$ by $\chi(X)$.
This is another corollary to Proposition \ref{prop:3}

\begin{Cor}
\label{cor:2}
In the situation of Proposition \ref{prop:3}, we have
 the formula $\chi(W_{{f}^{\prime}})=\chi(W_f)+(-1)^{n-k}\chi(S)$.
\end{Cor}

\subsection{Restrictions and flexibility on topological properties of maps and the source manifolds obtained by finite iterations of normal M-bubbling operations.}


We see restrictions and flexibility on topological properties of (the Reeb spaces of) maps and the source manifolds obtained by applying normal M-bubbling (S-bubbling) operations. We investigate the Euler numbers and homology groups of the Reeb spaces and source manifolds. 
  More precisely, we construct new fold maps from given explicit fold maps by applying normal M-bubbling (resp. S-bubbling) operations and observe that there are some topological restrictions on obtained maps and their Reeb spaces and the source manifolds. Moreover, we also show that the Euler numbers and the homology groups of obtained Reeb spaces and source manifolds
  are often various.
 
Before considering normal M-bubbling and S-bubbling operations and see topological properties of the resulting maps and source manifolds, we introduce fundamental key propositions and a class of fold maps.
 
 The following proposition includes propositions in \cite{saekisuzuoka} and \cite{kitazawa3}.

\begin{Prop}
\label{prop:4}
Let $m$ and $n$ be integers satisfying $m>n \geq 1$. Let $M$ be a closed
 and connected orientable manifold of dimension $m$ and $N$ be a manifold without boundary.
  
Then, for a simple fold map $f:M \rightarrow N$ such that inverse images of regular values
 are always disjoint unions of almost-spheres and that indices of singular points are always $0$ or $1$, two induced homomorphisms
 ${q_f}_{\ast}:{\pi}_j(M) \rightarrow {\pi}_j(W_f)$, ${q_f}_{\ast}:H_j(M;R) \rightarrow H_j(W_f;R)$, and ${q_f}^{\ast}:H^j(W_f;R) \rightarrow H^j(M;R)$ are isomorphisms for $0 \leq j \leq m-n-1$ and for any ring $R$.
 
Furthermore, if the ring $R$ is a PID and the relation $m=2n$ holds, then the rank of $M$ is twice the rank of $W_f$. In
 addition, if $H_{n-1}(W_f;R)$, which is isomorphic to $H_{n-1}(M;R)$, is free, then they are also free. 
\end{Prop}


The last statement of Proposition \ref{prop:4} on the homology groups for the $m=2n$ are not discussed
 in the mentioned papers but we easily obtain them by virtue of universal coefficient
  theorem and Poincare duality theorem. 

 Moreover, isomorphisms between the (co)homology groups or modules are not
 discussed in these papers. However, it is easy to know
 that the induced maps give mentioned isomorphisms. 

By using this proposition, we see
 restrictions and flexibility on algebraic topological properties of maps as in Example \ref{ex:1} (\ref{ex:1.2}) and the source manifolds. Most of obtained
  propositions and theorems of the present paper are stated as the statements on Reeb
   spaces and we can interpret them as ones on the source manifolds in the case of Example \ref{ex:1} (\ref{ex:1.2}). 

In addition, we review several fundamental facts on special generic maps easily understood by fundamental facts and discussions in \cite{saeki2} and we define a class of fold maps, which generalizes the class of special generic maps.

\begin{Def}
\label{def:4}
A fold map is said to be {\it almost special generic} if a fold map is obtained by an iteration of normal S-bubbling operations such that the homology groups of each generating manifold with coefficient ring $\mathbb{Z}$ are free from a special generic map.
\end{Def}

A special generic map is also an almost special generic map. 
The Reeb space of a special generic map into an Euclidean space is a compact parallelizable manifold with non-empty boundary whose dimension is equal to the dimension of the Euclidean space and that of an almost special generic map is in general not a manifold, which is a fundamental fact.  

We mention another fundamental fact on source manifolds of special generic maps. Every $m$-dimensional manifold represented as a
 connected sum of products of two standard spheres $S^{j_1}$ and $S^{j_2}$ admits a special generic map into ${\mathbb{R}}^n$ ($n \geq 2$) if the relation $1 \leq j_1 \leq n-1$ holds. Every $m$-dimensional manifold represented as a
 connected sum of products of two standard spheres $S^{j_1}$ and $S^{j_2}$ admits a special generic map into ${\mathbb{R}}^n$ if the relation $1 \leq j_1 \leq n$ holds. As an example, see also FIGURE \ref{fig:6}.
 
\begin{figure}
\includegraphics[width=80mm]{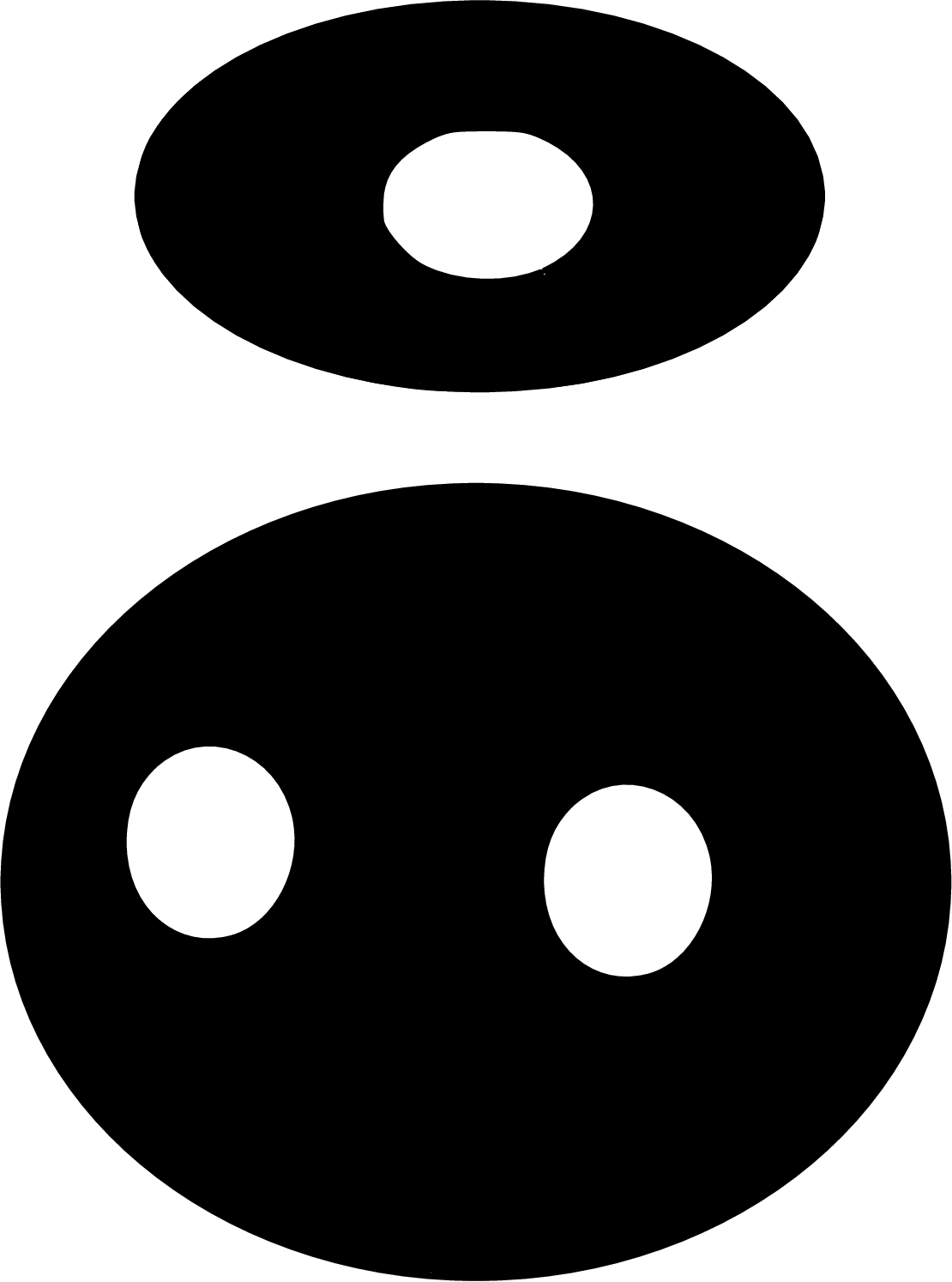}
\caption{Images of special generic maps on $S^1 \times S^{m-1}$ ($m \geq 3$) and a manifold represented as a connected sum of two copies of the manifold into the plane.}
\label{fig:6}
\end{figure}
In general, if an $m$-dimensional manifold admits a special generic map into ${\mathbb{R}}^n$, then a manifold represented as the connected sum of the previous manifold and a source manifold as in Proposition \ref{prop:2} admits an almost special generic map into ${\mathbb{R}}^n$. As an example, see also FIGURE \ref{fig:7}: note that the source manifold does not admit a special generic map into the plane by a classification theorem of manifolds admitting special generic maps into the plane in \cite{saeki2}.
\begin{figure}
\includegraphics[width=80mm]{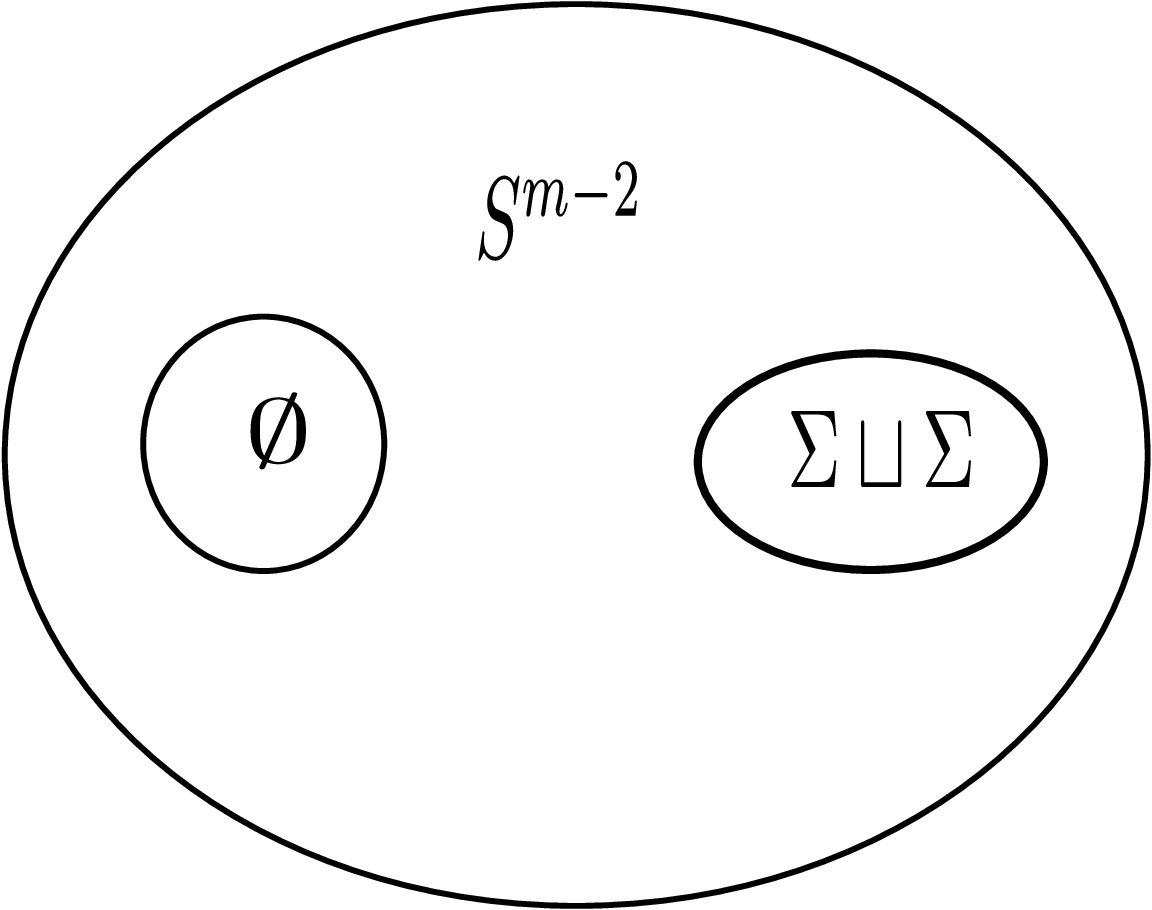}
\caption{The image of an almost special generic map on a manifold represented as a connected sum of $S^1 \times S^{m-1}$ ($m \geq 3$) and $S^2 \times \Sigma$ into the plane: manifolds and the empty set $\emptyset$ represent inverse images of corresponding regular values where $\Sigma$ is an almost-sphere.}
\label{fig:7}
\end{figure}

The following proposition follows essentially from facts on the restrictions on the homology groups of the Reeb spaces of special generic maps shown by Nishioka \cite{nishioka} together with Proposition \ref{prop:4}. We introduce the proposition without proof. 

\begin{Prop}
\label{prop:5}
Let $m$ and $n$ be integers satisfying $m-(n-2)>n\geq 3$.
If an $m$-dimensional manifold $M$ admits a special generic f into $N={\mathbb{R}}^n$. If the 1st homology groups $H_1(M;\mathbb{Z})$ and $H_1(W_f;\mathbb{Z})$ are trivial, then the {\rm (}$n-2${\rm )}-th homology groups $H_{n-2}(M;\mathbb{Z})$ and $H_{n-2}(W_f;\mathbb{Z})$ are free. 
\end{Prop}



Based on the preliminary, we now introduce some results.
First, applying Corollary \ref{cor:2}, we show the following.

\begin{Thm}
\label{thm:1}
Let $f:M \rightarrow N$ be a fold map.
\begin{enumerate}
\item
\label{thm:1.1}
 Let $n$ be even {\rm (}odd{\rm )}. If we perform a finite iteration of normal M-bubbling operations such that the Euler numbers of generating
 manifolds are not negative starting from $f$, then the resulting Euler numbers of the resulting Reeb spaces are not smaller {\rm (}resp. larger{\rm )} than the Euler number of $W_f$. Moreover, every positive integer not smaller {\rm (}resp. larger{\rm )} than this Euler number is realized as the resulting Reeb space of an iteration of normal M-bubbling operations.  
\item
\label{thm:1.2}
 Let $n$ be $1$ {\rm (}$2${\rm )}. If we perform a finite iteration of normal M-bubbling operations starting from $f$, then the resulting Euler number of the resulting Reeb space is not larger {\rm (}resp. smaller{\rm )} than the Euler number of $W_f$. 
\item
\label{thm:1.3}
 Let $n \geq 3$. By performing a finite iteration of normal M-bubbling {\rm (}S-bubbling{\rm )} operations starting from $f$, every integer is realized as the Euler number of the resulting Reeb space.
\end{enumerate}
\end{Thm}
\begin{proof}
The former part of the first statement immediately follows
 from Corollary \ref{cor:2} by considering normal M-bubbling {\rm (}S-bubbling{\rm )} surgeries whose generating manifolds are points, whose Euler numbers are $1$ (we take the points as in Example \ref{ex:1} (\ref{ex:1.1})). The latter part is easily obtained by applying a finite iteration of
 normal M-bubbling {\rm (}resp. S-bubbling{\rm )} operations such that the generating
 manifolds are points starting from $f$, by virtue of Corollary \ref{cor:2}.
  
 We easily have the second statement for $n=1$ by the fact that
  the generating manifold is always a point and by Corollary \ref{cor:2}. For $n=2$, we
   have the result by the fact that
  the generating manifold is always a point or a circle (the Euler number is $0$) and by Corollary \ref{cor:2}.
  
  We prove the third statement. We prove this in the case where $n$ is even. It is sufficient to show that
   for any integer $l$ smaller than the Euler number $\chi(W_f)$ of $W_f$, by a finite iteration of normal M-bubbling
 {\rm (}resp. S-bubbling{\rm )} operations starting from $f$, we obtain a fold map such that the resulting Euler number
    of the resulting Reeb space is $l$. We can take a closed, connected and orientable surface
     of genus $g \geq 2$, whose Euler number is $2-2g<0$, so that $\chi(W_f)+(2-2g)$ is not larger
      than $l$. We can also take $l-(\chi(W_f)+2-2g)$ distinct points which are not in the previous
       surface. By the finite iteration of normal M-bubbling {\rm (}resp. S-bubbling{\rm )} operations whose generating manifolds are homeomorphic to this surface and these points, we obtain a fold map such that the resulting number
    of the resulting Reeb space is $l$. Note that the surface and points can be taken as in Example \ref{ex:1} (\ref{ex:1.1}) or in an open ball in an appropriate connected component of the regular value set. We can prove the statement similarly in the case where $n$ is odd. 
    
\end{proof}

Next, by using Proposition \ref{prop:3} inductively, we have the following.

\begin{Thm}
\label{thm:2}
Let $R$ be a PID. Let $f:M \rightarrow N$ be a fold
 map.
 

\begin{enumerate}
\item
\label{thm:2.1}
 For any integer $0 \leq j \leq n$, we define $G_j$ as a free finitely generated module over $R$ so
  that $G_0$ is a trivial module and $G_n$ is not
 a trivial module, and the sum ${\sum}_{j=1}^{n-1} {\rm rank} \quad G_j$ of the ranks
 of $G_j$ is not larger than the rank of $G_n$. Then, by a finite iteration
 of normal M-bubbling {\rm (}S-bubbling{\rm )} operations starting from $f$, we obtain a fold map $f^{\prime}$ and $H_j(W_{{f}^{\prime}};R)$ is isomorphic to $H_j(W_f;R) \oplus G_j$. 

\item
\label{thm:2.2}
 Let $f^{\prime}$ be a fold map obtained by a finite iteration of
 normal M-bubbling operations starting from $f$. Then, $H_0(W_{f^{\prime}};R)$ is
 $R$ and $H_n(W_{f^{\prime}};R)$ is represented as the product sum of the original group $H_n(W_f;R)$ and a finitely generated and free module $G$.

Suppose that for a positive integer $k$ and for the Reeb
 space $W_{{f}^{\prime}}$, $H_j(W_{{f}^{\prime}};R)$ is isomorphic to $H_j(W_f;R)$ for $0<j<k$ and that in the case of $H_k(W_{{f}^{\prime}};R)$, this fact does not hold, then
 $H_k(W_{{f}^{\prime}};R)$ is represented as a product sum of $H_k(W_f;R)$ and a finitely generated and free module $G^{\prime}$ and the rank of $G^{\prime}$ is
 not larger than that of $G$.
  Especially, if these ranks coincide, then all generating manifolds of normal S-bubbling operations performed are {\rm (}$n-k${\rm )}-dimensional and two modules $G_j$ and ${G_j}^{\prime}$ such that 
  the product sum of $H_{n-k+j}(W_f;R)$ and $G_j$ is isomorphic to $H_{n-k+j}(W_{{f}^{\prime}};R)$ and
   that the product sum of $H_{n-j}(W_f;R)$ and ${G_j}^{\prime}$ is isomorphic to $H_{n-j}(W_{{f}^{\prime}};R)$ coincide for any integer $0 \leq j \leq k$.  
\item
\label{thm:2.3}
 Let $f^{\prime}$ be a fold map obtained by a finite iteration of
 normal M-bubbling operations starting from $f$. Then, $H_{n-1}(W_{f^{\prime}};R)$
 is represented as a module $H_{n-1}(W_f;R)$ and a finitely generated and free module.
\item
\label{thm:2.4}
Let $n \geq 3$ and odd {\rm (}even{\rm )}. Let $f^{\prime}$ be a fold map obtained by a finite iteration of
 normal M-bubbling operations such that the dimensions of the generating
  manifolds are smaller than $4$ starting from $f$. Assume that the Euler number
   $\chi(W_{{f}^{\prime}})$ is smaller {\rm (}resp. larger{\rm )} than the original Euler
    number $\chi(W_{f})$. Let $l$
    be the minimal integer not smaller than $\frac{|\chi(W_{{f}^{\prime}})-\chi(W_{f})|}{2}$.
     Then the
    rank of $H_n(W_{{f}^{\prime}};R)$ minus that of $H_n(W_{{f}^{\prime}};R)$ must be not smaller than $l$.
\item
\label{thm:2.5}
Let $n \leq 6$ and odd {\rm (}even{\rm )}. Let $f^{\prime}$ be a fold map obtained by a finite iteration of
 normal M-bubbling operations. If the difference $\chi(W_{{f}^{\prime}})-\chi(W_{f})$ is an odd number and smaller {\rm (}resp. larger{)\rm } than $-1$ ({\rm (}resp. 1{)\rm }), then, the
    rank of $H_n(W_{{f}^{\prime}};R)$ minus that of $H_n(W_{f};R)$ must be larger than $1$.
\end{enumerate}   
\end{Thm}
\begin{proof}
  We prove each statement in the case that the map $f$ is the canonical projection of the standard sphere or more precisely, in the case where the modules $H_{\ast}(W_f;R)$ and $ H_{\ast}(D^n;R)$ are isomorphic.
  
  We prove the first statement. Set $g_j:={\rm rank} G_j$. We can
 choose a disjoint union of $g_j$ copies of the ($n-j$)-dimensional standard
 spheres for $0<j \leq n$ smoothly embedded in an open disc in the interior ${\rm Int} f(M)$ of
 the image and $g_n-{\sum}_{j=1}^{n-1} g_j$ points in ${\rm Int} f(M)$. We apply normal M-bubbling  operations
 such that the generating
 manifolds are inverse images of spheres or points in the previous disjoint
 union by the maps from the Reeb spaces to ${\mathbb{R}}^n$ as in Definition \ref{def:2} one after
 another starting from the given map $f$; in fact we can apply strongly trivial normal M-bubbling operations.
Thus, by virtue of Proposition \ref{prop:3}, we can obtain a desired map.
 If we choose a disjoint union of $g_j$ ($n-j$)-dimensional
 manifolds whose homology types with coefficient rings $R$ are
 isomorphic to that of an ($n-j$)-dimensional
 sphere for $0<j<n$ smoothly embedded in the interior ${\rm Int} f(M)$ and $g_n-{\sum}_{j=1}^{n-1} g_j$ points there so that we can demonstrate a finite iteration of normal M-bubbling operations, then we always obtain
 a desired map.
  We prove the second statement. $H_0(W_{f^{\prime}};R)$ is
 $R$ and $H_n(W_{f^{\prime}};R)$ is finitely generated and free; by
 Proposition \ref{prop:3}, this easily follows and the rank
 of $H_n(W_{f^{\prime}};R)$ equals the number of the time of normal M-bubbling operations we need to obtain $f^{\prime}$.
By Proposition \ref{prop:3}, the maximum of the dimensions
 of the generating manifolds of these normal M-bubbling operations must be $n-k$. It
 also follows that $H_k(W_{{f}^{\prime}};R)$ is free and that its rank
 and the number of the time of normal M-bubbling operations with generating manifolds of
 dimension $n-k$ coincide. Thus the rank is not larger than the ran of $H_n(W_{f^{\prime}};R)$. \\
 The extra statements are obvious. In fact, if these ranks coincide, then all the generating manifolds of normal M-bubbling operations performed must be {\rm (}$n-k${\rm )}-dimensional and the ranks
   of $H_{n-k+j}(W_f;R)$ and $H_{n-j}(W_f;R)$ coincide for any integer $0 \leq j \leq k$ by virtue of Poincare duality theorem for each generating manifold and Proposition \ref{prop:3}.  \\
  We prove the third statement. Let $M_1$ and $M_2$ be $m$-dimensional, smooth, closed and connected manifolds, $N$ be an $n$-dimensional smooth manifold without boundary and $f_1:M_1 \rightarrow N$ be
 a fold map. Let $f_2:M_2 \rightarrow N$ be a fold map obtained
 by a normal M-bubbling operation to $f_1$ whose generating manifold is $S$ and of dimension $k$. By Proposition
 \ref{prop:3}, we
 have $H_{n-1}(W_{f_2};R) \cong H_{n-1}(W_{f_1};R) \oplus (H_{k-1}(S;R) \otimes R)$. Note
 that $H_{k-1}(S;R)$ is free; if $k$ is $1$ or $2$, then $S$ is, by the assumption, orientable
 and $H_{k-1}(S;R)$ is finitely generated and free, and if $k$ is larger
 than $2$, then $H^1(S;R)$ is finitely generated and free and by virtue of the Poincare
 duality theorem together with the assumption that $S$ is orientable, so is $H_{k-1}(S;R)$. By the induction, we have the mentioned result. 
 
 We prove the fourth statement. We prove this in the case
  where $n$ is odd. If the Euler number of a connected closed orientable manifold
   of dimension smaller than $4$ is not zero, then the manifold is a point
    or $2$-dimensional and the Euler number is $1$ or represented as $2-2g$ for a non-negative integer $g$, respectively. The minimum of
     the rank of $H_n(W_{{f}^{\prime}};R)$ is realized
      by considering a finite iteration of normal M-bubbling operations whose generating
       manifolds are a $2$-dimensional sphere in the case where the difference of Euler numbers $\chi(W_{{f}^{\prime}})-\chi(W_{f})$ is even and the rank is $\frac{|\chi(W_{{f}^{\prime}})-\chi(W_{f})|}{2}$. In the case where the difference is odd, the minimum of the rank is realized by considering a finite iteration of normal S-bubbling operations such that the generating manifold of one of the operations is a point and that the generating manifolds of the others are $2$-dimensional spheres and the rank is $\frac{|\chi(W_{{f}^{\prime}})-\chi(W_{f})|-1}{2}+1$. We can
          prove the statement in the case where $n$ is even similarly.
          
  We prove the last statement. 
  The Euler numbers of odd dimensional closed and connected orientable manifolds are $0$. The Euler numbers of even dimensional closed and connected orientable manifolds we can embed into manifolds whose dimensions are not larger than $6$ must be even numbers or odd numbers which are smaller than $2$. The last part of this fact on Euler numbers of even dimensional closed and connected orientable manifolds is shown by the following.   
   By a fundamental discussion on characteristic classes or Stiefel-Whitney classes, $4$-dimensional manifolds we can embed into $5$ or $6$-dimensional manifolds are spin and their 2nd Stiefel-Whitney classes vanish and the Euler numbers are always even. $0$-dimensional closed and connected orientable manifolds are points and the Euler numbers are $1$ and for $2$-dimensional ones, the Euler numbers are even and not larger than $2$.           
   By virtue of Corollary \ref{cor:2} and these facts on embeddings of manifolds, by a single M-bubbling operation or equivalently, in the case
    where the
    rank of $H_n(W_{{f}^{\prime}};R)$ minus that of $H_n(W_{f};R)$ is $1$, for the resulting fold map ${f}^{\prime}$, the value $\chi(W_{{f}^{\prime}})-\chi(W_{f})$ cannot be an odd number not larger than $-3$ in the case where $n$ is odd and cannot be an odd number not smaller than $3$.

  This completes the proof.   
\end{proof}

Let us see cohomology groups. To simplify the problem, let us assume that each module $H_k(W_f;R)$ of the original Reeb space $W_f$ is free. 
 By universal coefficient theorem, in the situation of Theorem \ref{thm:2} (\ref{thm:2.1}),
 $H^j(W_{f^{\prime}};R)$ is isomorphic to $H_j(W_{f^{\prime}};R)$ for $0 \leq j \leq n$ and in the situation
 of (\ref{thm:2.2}) and (\ref{thm:2.3}), $H^j(W_{f^{\prime}};R)$
 is isomorphic to $H_j(W_{f^{\prime}};R)$ for $0 \leq j \leq k$ and $j=n$ and $H^{k+1}(W_{f^{\prime}};R)$ is isomorphic to ${\rm Hom}_{R}(H_{k+1}(W_{f^{\prime}};R),R)$. 

\begin{Ex}
\label{ex:2}
\begin{enumerate}
\item \label{ex:2.1}
We can say that in \cite{kitazawa} and \cite{kitazawa2}, maps
 in Proposition \ref{prop:2} have been obtained by finite iterations of (strongly) trivial normal
 S-bubbling operations
 whose generating manifolds are points starting from a round fold map whose singular set is connected on the standard sphere of dimension $m$ into the ${\mathbb{R}}^n$. Note that the Reeb
 spaces are simple homotopy equivalent to the bouquet of $l-1$ copies of $S^n$. This is a specific case of Theorem \ref{thm:2} (\ref{thm:2.1}). 

In \cite{kobayashi2}, Kobayashi invented a {\it bubbling
 surgery} to a stable (fold) map and we can also say that
 maps in Proposition \ref{prop:2} have been obtained by finite iterations of
 bubbling surgeries.   

More precisely, a {\it bubbling surgery} is a surgery operation
 to construct a new stable fold
 map from a given stable fold map so that the singular value set of the new map
 is the disjoint union of that of the given map and a sphere bounding a closed disc whose
 interior does not contain singular values and that the outside of a small neighborhood
 of the disc does not change. For this, see FIGURE \ref{fig:2} and FIGURE \ref{fig:3}, some of figures in which
 represent bubbling surgeries explicitly. Note that a bubbling surgery is a normal bubbling operation whose generating manifold is a point.    
\item \label{ex:2.2}
In the situation of Theorem \ref{thm:2} (\ref{thm:2.3}), if $m=2n$ holds, each module $H_j(W_f;R)$ is free and $f$ is a fold map satisfying the condition in Proposition \ref{prop:4} posed on the map, then by the proposition, $H_j(M;R)$ and $H_j(W_f;R)$ are isomorphic for $0 \leq j \leq m-n-1=n-1$ and $H_n(M;R)$ is an R-module whose rank is twice the rank of $H_n(W_f;R)$. In addition, $H_{n-1}(W_f;R)$, which is isomorphic to $H_{n-1}(M;R)$, is free, so the two modules $H_n(M;R)$ and $H_n(W_f;R)$ are also free.
\end{enumerate}
\end{Ex}


  For example, Theorem \ref{thm:1} (\ref{thm:1.1}) and (\ref{thm:1.3}) and Theorem \ref{thm:2} (\ref{thm:2.1}) show flexibility of Euler numbers and homology groups of Reeb spaces and as a result those of source manifolds of the resulting maps explicitly.

  The following theorem also shows such flexibility.

\begin{Thm}
\label{thm:3}

Let $k$ and $n$ be positive integers and let $l$ be a non-negative integer satisfying $l<l+1<k<n$. 
Let $S$ be a closed, connected and orientable manifold of dimension
 $n-k$ which we can embed into ${\mathbb{R}}^{n-l}$. 

Let $R$ be a PID. Let $f:M \rightarrow N$ be a fold map

Then, by a normal M-bubbling {\rm (}S-bubbling{\rm )} operation to $f$, we have
 a fold map $f^{\prime}:M^{\prime} \rightarrow N$ satisfying the following condition.
 
$H_j(W_{f^{\prime}};R) \cong \begin{cases} R \quad (j=0) \\
H_j(W_{f};R) \oplus H_{j-l-1}(S;R) \quad (l+1 \leq j \leq k-1) \\
H_j(W_{f};R) \oplus H_{j-l-1}(S;R) \oplus H_{j-k}(S;R) \quad (k \leq j \leq n-1) \\
H_n(W_{f};R) \oplus R \quad (j=n) \\
H_j(W_{f};R) \quad (otherwise)
\end{cases}      
$
\end{Thm}
\begin{proof}
We can embed $S$ into ${\mathbb{R}}^{n-l}$ and there exists a manifold $S^{\prime}$ regarded as the
 total space of an oriented linear $S^{k-l-1}$-bundle over $S$ whose
 Euler class vanishes and which we can embed into ${\mathbb{R}}^n$.
 For any PID $R$, we have  

$H_j(S^{\prime};R) \cong \begin{cases} H_j(S;R) \quad (0 \leq j \leq k-l-2) \\
H_j(S;R) \oplus H_{j-k+l+1}(S;R)  \quad (k-l-1 \leq j \leq n-l-2) \\
H_j(S;R) \quad (j=n-l-1)
\end{cases}      
$.

By virtue of Proposition \ref{prop:3}, by a normal M-bubbling (S-bubbling) operation to $f$ such
 that the generating manifold is ${\bar{f}}^{-1}(S^{\prime})$, we obtain a desired map $f$; for example, take $S$ in an open ball in an appropriate connected component of the regular value set as in Example \ref{ex:1} (\ref{ex:1.1}) and perform a trivial normal M-bubbling (resp. S-bubbling) operation.
\end{proof}

As a specific case, we have the following.

\begin{Cor}
\label{cor:3}

Let $S$ be a closed, connected and orientable manifold whose dimension
 is not smaller than $1$ and not larger than $\frac{n}{2}$. Set $l$ be
 an integer satisfying $0 \leq l \leq n-2\dim S$ and if $\dim S=1$ holds or $S$ is a circle, then
 assume that $0 \leq l <n-2\dim S$ holds.

Let $R$ be a PID. Let $f:M \rightarrow N$ be a fold map. 

Then, by a normal M-bubbling {\rm (}S-bubbling{\rm )} operation to $f$, we have
 a fold map $f^{\prime}:M^{\prime} \rightarrow N$
satisfying the following.

$H_j(W_{f^{\prime}};R) \cong \begin{cases} 
R \quad (j=0) \\
H_j(W_f;R) \oplus H_{j-l-1}(S;R) \quad (l+1 \leq j \leq n-\dim S-1) \\
H_j(W_f;R) \oplus H_{j-l-1}(S;R) \oplus H_{j-n+\dim S}(S;R)\\
 (n-\dim S \leq j \leq n-1) \\
H_n(W_f;R) \oplus R \quad (k=n) \\
H_j(W_f;R) \quad (otherwise)
\end{cases}      
$.

\end{Cor}

\begin{proof}
We can embed any closed manifold X into ${\mathbb{R}}^{2 \dim X}$ by virtue of Whitney embedding
 theorem. From this fact, in the situation
 of Theorem \ref{thm:4}, we can
 set $n-k:=\dim S$ and we have $l<l+1<k<n$. Thus, we have the result.

\end{proof}

\begin{Ex}
\label{ex:3}
Let $G$ be any finitely generated commutative group. There
 exists a closed, connected and orientable $3$-dimensional
 manifold $S$ satisfying $H_1(S;\mathbb{Z})=G$. We can
 embed $S$ into ${\mathbb{R}}^n$ where $n \geq 5$ holds
 (\cite{wall}). For such $n$, we can apply Proposition \ref{prop:3} to a fold map $f:M \rightarrow N$ such that for the resulting
 map $f^{\prime}:M^{\prime} \rightarrow N$, $H_{j}(W_{f^{\prime}};\mathbb{Z})$ is isomorphic to
 $H_{j}(W_{f^{\prime}};\mathbb{Z})$ for $0<j \leq n-4$
 and $H_{n-3}(W_{f^{\prime}};\mathbb{Z})$ is isomorphic to the direct sum of $H_{n-3}(W_{f};\mathbb{Z})$
 and $\mathbb{Z}$ and $H_{n-2}(W_{f^{\prime}};\mathbb{Z})$ is isomorphic to the direct sum of
  $H_{n-2}(W_{f};\mathbb{Z})$ and $G$. By choosing suitable integers $k \geq 1$ and $l \geq 0$, we can
 also apply Theorem \ref{thm:2} or, for $n \geq 6$, Corollary \ref{cor:1}. In this case, for the map $f^{\prime}$, $H_{l+1}(W_{f^{\prime}};\mathbb{Z})$ is
 isomorphic to $H_{l+1}(W_{f};\mathbb{Z})$ and $\mathbb{Z}$ and $H_{l+2}(W_{f^{\prime}};\mathbb{Z})$
 is isomorphic to the direct sum of $H_{l+2}(W_{f};\mathbb{Z})$ and $G$ or $G \oplus \mathbb{Z}$. 
\end{Ex}

  Related to this example and the presented theorem and corollary, we show explicit result on flexibility of homology groups of Reeb spaces of maps obtained by normal M-bubbling (S-bubbling) operations whose generating manifolds are as mentioned in the example or the theorem and corollary.
  
Let $n$ be an integer larger than $6$ and let $f:M \rightarrow N$ be a fold map.
 Let $j$ be an integer satisfying $0 \leq j \leq n-7$ and let $G_j$ be a finite
 commutative group. Let $\{g_j\}_{j=0}^{n-7}$ be a sequence of integers satisfying the following.
 \begin{itemize}
 \item All the integers are not smaller than $-1$.
 \item If $g_j=-1$, then $G_j$ is a trivial group.
 \end{itemize} 
 
  There exists a $3$-dimensional closed and connected orientable
  manifold $S_j$ satisfying $H_1(S_j;\mathbb{Z}) \cong G_j$, $H_2(S_j;\mathbb{Z}) \cong \{0\}$
   and $H_0(S_j;\mathbb{Z}) \cong H_3(S_j;\mathbb{Z})=\mathbb{Z}$ (for example consider connected sums of Lens spaces). If $g_j>-1$ holds, then we
   can choose the family of disjoint submanifolds in $N$ satisfying the following two.
   \begin{enumerate}
   \item The family includes a manifold diffeomorphic to $S_j \times S^{n-j-4}$.
   \item The number of manifolds whose dimensions are $\dim (S_j \times S^{n-j-4})$ is $g_j+1$ and except the manifold diffeomorphic to $S_j \times S^{n-j-4}$, homology types of these manifolds with coefficient rings $\mathbb{Z}$ are same as that of a sphere.
   \end{enumerate} 
  We also define a sequence $\{h_j\}_{j=1}^n$ of non-negative integers satisfying the following.
 \begin{itemize}
 \item if $g_j \geq 0$ ($0 \leq j \leq n-7$), then $h_{j+1}$ is $0$.
 \item $h_n$ is not smaller than the sum $\sum_{j_1}^{n-1} h_j$
 \end{itemize} 
   
   We can choose a family of disjoint submanifolds in $N$ satisfying the following.
\begin{itemize}
\item As ($n-j$)-dimensional manifolds, the family include just $h_j$ manifolds whose homology groups with coefficients $\mathbb{Z}$ are isomorphic to
 that of an ($n-j$)-dimensional sphere for $j \neq n-3,n$. 
\item As ($n-j$)-dimensional manifolds, the family include just $h_j$ manifolds whose homology groups with coefficient rings $\mathbb{Q}$ are isomorphic to  
 that of an ($n-j$)-dimensional sphere or manifolds whose homology groups with coefficient rings $\mathbb{Z}$ are
  represented as $3$-dimensional manifolds $S_k$ ($0 \leq k \leq n-7$) just before for $j=n-3$. 
\end{itemize}   
   
   We can take each submanifold here in an open ball in an appropriate connected component of the regular value set as explained in the presentation of a strongly trivial normal bubbling operation in Example \ref{ex:1} (\ref{ex:1.2}).

   By demonstrating normal M-bubbling (S-bubbling) operations, for example, strongly trivial normal M-bubbling (resp. S-bubbling) operations, such that the generating manifold is the inverse images of each submanifolds by the natural maps from the Reeb spaces into
    $N$ one after another and applying Proposition \ref{prop:3}, we obtain a new map $f^{\prime}:M^{\prime} \rightarrow N$ as in the following theorem. 

\begin{Thm}
\label{thm:4}
Let $m$ and $n$ be positive integers satisfying $m>n \geq 7$. Let $f:M \rightarrow N$ be a fold map

For integers $0 \leq j \leq n-7$, we introduce integers $g_j$ and finite
 commutative groups $G_j$ satisfying the following two.
 \begin{enumerate}
 \item $g_j \geq -1$.
 \item If $g_j=-1$, then $G_j$ is a trivial group.
 \end{enumerate}

For any integer $1 \leq j \leq n$, we also introduce non-negative integers $h_j$ satisfying the following.
\begin{enumerate}
\item If $g_j \geq 0$ {\rm (}$0 \leq j \leq n-7${\rm )}, then $h_{j+1}$ is $0$.
\item $h_n$ is not smaller than the sum $\sum_{j_1}^{n-1} h_j$.
\end{enumerate} 

Let $H$ be a finite commutative group and set groups $\{H_j\}_{j=1}^{n}$ as
\begin{itemize}
\item $H_j$ is isomorphic to $H_j(W_f;\mathbb{Z}) \oplus {\mathbb{Z}}^{h_j}$ for $j \neq n-2$
\item $H_j$ is isomorphic to $H_j(W_f;\mathbb{Z}) \oplus {\mathbb{Z}}^{h_j} \oplus H$ for $j=n-2$.
\end{itemize}
By a finite iteration of {\rm (}strongly trivial{\rm )} normal M-bubbling {\rm (}S-bubbling{\rm )} operations such that the generating manifold of each operation
  is a product of a $3$-dimensional manifold and a standard sphere or a manifold having the same homology type with coefficient ring $\mathbb{Z}$ as that of a sphere starting from $f$, we obtain a map $f^{\prime}:M^{\prime} \rightarrow N$ satisfying the following;
   
 $$H_i(W_{{f}^{\prime}};\mathbb{Z}) \cong 
 \begin{cases}
 \mathbb{Z} \quad {\rm (}i=0{\rm )} \\
 H_i \oplus {\mathbb{Z}}^{g_0+1} \quad {\rm (}i=1{\rm )}\\
 H_i \oplus {\mathbb{Z}}^{g_{i-1}+1} \oplus G_{i-2} \quad {\rm (}i=2,3{\rm )}\\
 H_i \oplus {\oplus}_{j=0}^{n-7} G_j \quad {\rm (}i=n-2{\rm )} \\ 
 H_i \quad {\rm (}i=n-1{\rm )} \\
 H_i \oplus {\mathbb{Z}}^{({\sum}_{j=0}^{n-7} g_j)+n-6} \quad {\rm (}i=n{\rm )} \\
 H_i \oplus  {\mathbb{Z}}^{g_{i-1}+2} \oplus G_{i-2} \quad {\rm (}4 \leq i \leq n-6,g_{i-4}>-1{\rm )}\\ 
 H_i \oplus {\mathbb{Z}}^{g_{i-1}+1} \oplus G_{i-2} \quad {\rm (}4 \leq i \leq n-6,g_{i-4}=-1{\rm )}\\
 H_i \oplus {\mathbb{Z}} \oplus G_{n-7} \quad {\rm (}4 \leq i=n-5, g_{n-9}>-1{\rm )} \\
 H_i \oplus G_{n-7} \quad {\rm (}4 \leq i=n-5, g_{n-9}=-1{\rm )} \\
 H_i \oplus {\mathbb{Z}} \quad {\rm (}4 \leq i=n-4,g_{n-8}>-1{\rm )} \\
 H_i \quad {\rm (}4 \leq i=n-4,g_{n-8}=-1{\rm )} \\
 H_i \oplus {\oplus}_{j \in \{j^{\prime}|g_{j^{\prime}}>-1\}} H_{n-j-4}(S_j \times S^{n-j-4};\mathbb{Z}) \quad {\rm (}i=n-3{\rm )}
\end{cases}$$.
\end{Thm}

We present precise presentations on the homology groups which are important in knowing the types of homology groups to lead us to the result of Theorem \ref{thm:4} and this gives a proof of this theorem. We abuse notation just before Theorem \ref{thm:4}.

We consider the case where the two groups $H_{*}(W_f;\mathbb{Z})$ and $H_{*}(D^n;\mathbb{Z})$ are isomorphic. In addition, we also assume that the group $H$ is trivial and $h_j=0$. Of course we can consider similarly in general cases and for general cases, we can obtain the result easily by Proposition \ref{prop:3} : moreover, if  $h_j$ and the groups $H_j$ and $H$ are general, then we need $h_j$ ($n-j$)-dimensional manifolds whose homology groups with coefficient rings $\mathbb{Z}$
 are isomorphic to that of an ($n-j$)-dimensional sphere for $j \neq n-3, n$, $h_n-\sum_{j_1}^{n-1} h_j$ points and and a family of $h_j$ 3-dimensional manifolds satisfying the following for $j=n-3$.
 
\begin{itemize}
\item If $h_j=0$, then the family is empty.
\item If $h_j>0$, then the family includes just one $3$-dimensional manifold such that the homology group with coefficient ring $\mathbb{Q}$ is isomorphic to the homology group with coefficient ring $\mathbb{Q}$ of a sphere and that the 1st homology group with coefficient ring $\mathbb{Z}$ is $H$, and for all the other manifolds of the family, the homology groups with coefficient rings $\mathbb{Z}$ are isomorphic to the homology group with coefficient ring $\mathbb{Z}$ of a sphere.    
\end{itemize}  

We have
$$H_1(W_{{f}^{\prime}};\mathbb{Z}) \cong 
\begin{cases}
 H_0(S_0 \times S^{n-4};\mathbb{Z}) \oplus {\mathbb{Z}}^{g_0} \quad {\rm (}g_0 \neq -1{\rm )}\\
\{0\} \quad {\rm (}g_0=-1{\rm )}
\end{cases}, $$
$$H_2(W_{{f}^{\prime}};\mathbb{Z}) \cong
 \begin{cases} 
 H_0(S_1 \times S^{n-5};\mathbb{Z})
 \oplus {\mathbb{Z}}^{g_1} \oplus H_1(S_0 \times S^{n-4};\mathbb{Z}) \quad {\rm (}g_0>-1,g_1>-1{\rm )}\\
  H_0(S_1 \times S^{n-5};\mathbb{Z}) \oplus {\mathbb{Z}}^{g_1} \quad {\rm (}g_0=-1,g_1>-1{\rm )}\\ 
 {\mathbb{Z}}^{g_1+1} \oplus H_1(S_0 \times S^{n-4};\mathbb{Z}) \quad {\rm (}g_0>-1,g_1=-1{\rm )}\\
 {\mathbb{Z}}^{g_1+1} \quad {\rm (}g_0=-1,g_1=-1{\rm )}\\
 \end{cases}
 $$

$$H_3(W_{{f}^{\prime}};\mathbb{Z}) \cong
 \begin{cases}
  H_0(S_2 \times S^{n-6};\mathbb{Z})
 \oplus {\mathbb{Z}}^{g_2} \oplus H_1(S_1 \times S^{n-3};\mathbb{Z}) \quad {\rm (}g_1>-1,g_2>-1{\rm )}\\
   H_0(S_2 \times S^{n-6};\mathbb{Z}) \oplus {\mathbb{Z}}^{g_2} \quad {\rm (}g_1=-1,g_2>-1{\rm )}\\
H_1(S_1 \times S^{n-3};\mathbb{Z}) \quad {\rm (}g_1>-1,g_2=-1{\rm )}\\
\{0\} \quad {\rm (}g_1=-1,g_2=-1{\rm )}
\end{cases}.$$

Let $\tilde{H_j}=H_{1+(n-j-4)}(S_j \times S^{n-j-4};\mathbb{Z}) \cong G_j$ for $g_j>-1$ and let $\tilde{H_j}=G_j \cong \{0\}$
 for $g_j=-1$. We have

$$H_{n-2}(W_{{f}^{\prime}};\mathbb{Z}) \cong {\oplus}_{j=0}^{n-7} \tilde{H_j} \cong {\oplus}_{j=0}^{n-7} G_j$$.

$2+(n-j-4)=n-j-2$ and the ($n-j-2$)-th homology group of each generating
 manifold vanishes. These facts give us $H_{n-1}(W_{{f}^{\prime}};\mathbb{Z}) \cong \{0\}$. 
 
 $H_{n}(W_{{f}^{\prime}};\mathbb{Z})$ is finitely generated free group and
  its rank equals the number of times of normal M-bubbling (S-bubbling) operations: the sum of $g_j+1$ for all $0 \leq j \leq n-7$. 
  For $4 \leq i \leq n-6$, we have the following;
  $$H_i(W_{{f}^{\prime}};\mathbb{Z}) \cong
  \begin{cases}
   H_0(S_{i-1} \times S^{n-i-3};\mathbb{Z}) \oplus {\mathbb{Z}}^{g_{i-1}}
    \oplus H_3(S_{i-4} \times S^{n-i};\mathbb{Z}) \oplus G_{i-2} \\
     \cong {\mathbb{Z}}^{g_{i-1}+2} \oplus G_{i-2} \quad {\rm (}g_{i-1}>-1,g_{i-4}>-1{\rm )}\\
H_3(S_{i-4} \times S^{n-i};\mathbb{Z}) \oplus G_{i-2} \\
\cong {\mathbb{Z}}^{g_{i-1}+2} \oplus G_{i-2} \quad {\rm (}g_{i-1}=-1,g_{i-4}>-1{\rm )}\\
 H_0(S_{i-1} \times S^{n-i-3};\mathbb{Z}) \oplus {\mathbb{Z}}^{g_{i-1}} \oplus G_{i-2} \\
 \cong {\mathbb{Z}}^{g_{i-1}+1} \oplus G_{i-2} \quad {\rm (}g_{i-1}>-1,g_{i-4}=-1{\rm )}\\
 G_{i-2} \cong {\mathbb{Z}}^{g_{i-1}+1} \oplus G_{i-2}  \quad {\rm (}g_{i-1}=-1,g_{i-4}=-1{\rm )}    
    \end{cases}.$$
    For $4 \leq i=n-5$, we have the following;
$$H_i(W_{{f}^{\prime}};\mathbb{Z}) \cong \begin{cases}
   H_3(S_{n-9} \times S^5;\mathbb{Z}) \oplus H_1(S_{n-7} \times S^3;\mathbb{Z}) \cong \\
   \mathbb{Z} \oplus G_{n-7}  \quad {\rm (}g_{n-9}>-1,g_{n-7}>-1{\rm )} \\
   H_1(S_{n-7} \times S^3;\mathbb{Z}) \cong G_{n-7}  \quad {\rm (}g_{n-9}=-1,g_{n-7}>-1{\rm )}\\
   H_3(S_{n-9} \times S^5;\mathbb{Z}) \cong \mathbb{Z} \oplus G_{n-7}  \quad {\rm (}g_{n-9}>-1,g_{n-7}=-1{\rm )}\\
   \{0\} \cong G_{n-7}  \quad {\rm (}g_{n-9}=-1,g_{n-7}=-1{\rm )}   
   \end{cases}.
   $$
   For $4 \leq i=n-4$, we have the following;
   $$H_i(W_{{f}^{\prime}};\mathbb{Z}) \cong \begin{cases}
   H_3(S_{n-8} \times S^4;\mathbb{Z}) \cong \mathbb{Z} \quad {\rm (}g_{n-8}>-1{\rm )}\\
   \{0\} \quad {\rm (}g_{n-8}=-1{\rm )}
   \end{cases}.
   $$
   
   Last, we have the following:
   $H_{n-3}(W_{{f}^{\prime}};\mathbb{Z})$ is represented as the direct sum of all
    the groups $H_{n-j-4}(S_j \times S^{n-j-4};\mathbb{Z})$ for $j$ such that $g_j>-1$. 

This completes the precise presentations of homology groups.

We can also find homology types of Reeb spaces of maps we cannot obtain by constructions of maps by normal M-bubbling (S-bubbling) operations as the following.
   
\begin{Thm}
\label{thm:5}
Let $m$ and $n$ be positive integers satisfying the relation $m>n$. Let $f:M \rightarrow N$ be a fold
 map such that for a sequence $\{G_j\}_{j=1}^{n}$ of finitely generated commutative groups
  , $H_j(W_{{f}^{\prime}};\mathbb{Z})$ is represented as $H_{j}(W_f;\mathbb{Z}) \oplus G_j$. 

Let us consider a finite iteration of normal M-bubbling {\rm (}S-bubbling{\rm )} operations starting from $f$ satisfying the following.
\begin{enumerate}
\item This iteration consists of $k$ normal M-bubbling {\rm (}resp. S-bubbling{\rm )} operations such that the
 generating manifold of the $j$-th operation is $S_j$ where the relation $j \leq k$ holds.   
\item Let $i_0$ be a positive integer. For each $S_{j_1}$ and for
 any homology group $H_{j_2}(S_{j_1};\mathbb{Z})$ isomorphic
 to a non-trivial finite group, then for some positive
  integer $i_j \leq i_0$, $H_{j_2-i_j}(S_{j_1};\mathbb{Z})$ {\rm (}$H_{j_2+i_j}(S_{j_1};\mathbb{Z})${\rm )} is an infinite group.   
\end{enumerate}
Then by the iteration of normal M-bubbling {\rm (}resp. S-bubbling{\rm )} operations starting from $f$, we have a fold
 map $f^{\prime}:M^{\prime} \rightarrow N$ and for
  any group $G_{j^{\prime}}$ isomorphic
 to a non-trivial finite group, there exists a positive
  integer ${j^{\prime}}_a \leq i_0$ and $G_{j^{\prime}-{{j}^{\prime}}_a}$ {\rm (}resp. $G_{j^{\prime}+{{j}^{\prime}}_a}${\rm )} is an infinite group.

  Furthermore, assume that there exists an integer $d$ satisfying the following.
  \begin{enumerate}
  \item For an integer $d_1$, the group $G_{d_1}$ is non-trivial and finite. 
  \item For any non-negative integer $d_2$ smaller than $d$, the group $G_{d_1+d_2}$ is non-trivial and finite.
  \end{enumerate}
  Then $d$ is not larger than $i_0$. 
\end{Thm} 
\begin{proof}
For any group $G_{{j^{\prime}}}$ in the sequence $\{G_j\}$ isomorphic
 to a non-trivial finite group, by Proposition \ref{prop:3}, there exists a generating manifold $S_c$, an
   integer ${j^{\prime}}_b > 0$ and a non-trivial finite subgroup represented
   as $H_{j^{\prime}-{j^{\prime}}_b}(S_c;\mathbb{Z}) \otimes \mathbb{Z}$; note that
    the dimension of $S_c$ is $n-{j^{\prime}}_b$ and
     that $H_{j^{\prime}-{j^{\prime}}_b}(S_c;\mathbb{Z})$ is a non-trivial finite group.
    
 We prove the first statement in the case where $H_{j_2-i_j}(S_{j_1};\mathbb{Z})$ is an infinite
  group for the integers $j_1$, $j_2$ and $i_j$ (we call this the {\it minus} case). Set $j_1=c$
  and $j_2=j^{\prime}-{j^{\prime}}_b$ in $H_{j^{\prime}-{j^{\prime}}_b}(S_c;\mathbb{Z}) \otimes \mathbb{Z}$ in the previous paragraph. Then for a positive
  integer $i_j \leq i_0$, $H_{j_2-i_j}(S_{c};\mathbb{Z})$ is an infinite group. As a result, we have $G_{j_2-i_j+{j^{\prime}}_b} \cong G_{j^{\prime}-{j^{\prime}}_b-i_j+{j^{\prime}}_b} \cong G_{j^{\prime}-i_j}$ and the group is infinite and $H_{j_2-i_j}(S_c;\mathbb{Z}) \otimes \mathbb{Z}$ is regarded as its
   subgroup by Proposition \ref{prop:3} since $S_c$ is ($n-{j^{\prime}}_b$)-dimensional. We can
    take $j^{\prime}_a=i_j$ and this completes the proof.
  
   We can prove this similarly in the case where $H_{j_2+i_j}(S_{j_1};\mathbb{Z})$ is
    an infinite group (we call this the {\it plus} case) for the integers $j_1$, $j_2$ and $i_j$.
  
  From these statements, the last statement for the integer $d$ follows immediately.
\end{proof}

\begin{Ex}
\label{ex:4}
We consider the iteration of normal M-bubbling operations in the proof of Theorem \ref{thm:4}. This case is an example of the case $i_0=1$ and the minus case in the situation of Theorem \ref{thm:5}. We can also say that this case is an example of the case $i_0=2$ and the plus case. 
\end{Ex}
 
We introduce another family of maps obtained by the operations from a given map.

\begin{Thm}
\label{thm:6}
Let $f:M \rightarrow N$ be a fold
 map and let $n \geq 5$.
 
 For any integer $0 \leq j \leq n-1$, we define $G_j$ as a finite commutative group so
  that $G_0$, $G_1$, $G_{n-2}$ and $G_{n-1}$ are trivial and that for $n \geq 6$, $G_{n-3}$ is regarded as the direct sum of the direct sum $\oplus_{j=2}^{n-4} G_j$ and a commutative group $G$.
 We also define a non-negative integer $h_j$ for any integer $0 \leq j \leq n-1$ so that the following hold.
 \begin{enumerate}
 \item If for each integer $2 \leq j \leq n-4$, $G_j$ is not trivial, then $h_{j-1}$ is positive.
 \item if the direct sum $\oplus_{j=2}^{n-4} G_j$ is not ismorphic to $G_{n-3}$, or $n=5$ and $G_{n-3}$
  is not trivial, then $h_{n-4}$ is positive.
 \item $h_{n-1}$ is not smaller than the sum ${\sum}_{j=0}^{n-2} h_j$.
 \end{enumerate} 
 Under these assumptions, by a finite iteration
 of M-bubbling {\rm (}S-bubbling{\rm )} operations starting from $f$, we obtain a fold map $f^{\prime}$ such that for $j>0$, $H_j(W_{{f}^{\prime}};R)$ is isomorphic to the group $H_j(W_f;R) \oplus {\mathbb{Z}}^{h_{j-1}} \oplus G_{j-1}$. 
\end{Thm}

\begin{proof}
As mentioned, for any commutative group $G$, there exists a $3$-dimensional connected, closed and orientable manifold $S_G$ such that the groups $H_0(S_G;\mathbb{Z})$ and $H_3(S_G;\mathbb{Z})$ are isomorphic to $\mathbb{Z}$, $H_1(S_G;\mathbb{Z})$ is isomorphic to $G$ and $H_2(S_G;\mathbb{Z})$ is trivial and this manifold can be smoothly embedded into ${\mathbb{R}}^5$. Note that by fundamental technique of differentiable topology or the theory of higher-dimensional knots and links, for any positive integer $k$, we can construct a ($k+3$)-dimensional closed manifold $S_{G,k}$ by gluing $D^{k+1} \times S^2$ and the product of $S^k$ and a manifold obtained by removing the interior of a smoothly embedded $3$-dimensional closed disc from $S_G$ on the boundaries so that the following hold.
\begin{enumerate}
\item For $j=1$ and $j=k+1$, $H_j(S_{G,k};\mathbb{Z})$ is isomorphic to $G$ and for any
 other integer $0<j<k+3$, $H_j(S_{G,k};\mathbb{Z})$ is trivial.
\item We can smoothly embed $S_{G,k}$ into ${\mathbb{R}}^{k+5}$.
\end{enumerate}
We need the following family of generating manifolds to obtain the resulting map and source manifold by normal M-bubbling  {\rm (}S-bubbling{\rm )} operations.
\begin{enumerate}
\item If for any integer $1 \leq j \leq n-4$, then there exist just $h_j$ ($n-j-1$)-dimensional manifolds. If $h_j>0$ holds, then there exists $h_j-1$ manifolds whose homology groups are isomorphic to $S^{n-j-1}$ 
for $1 \leq j \leq n-4$. Moreover, for each integer $1 \leq j < n-4$, if $h_j>0$ holds, then
 just one manifold diffeomorphic to $S_{G_j,n-j-4}$ exists in the family and for $j=n-4$, if $h_j>0$ holds, then exists just one manifold diffeomorphic to $S_G$ in the family.
\item For $j=0,n-3,n-2$, there exist just $h_j$ ($n-j-1$)-dimensional manifolds and their homology groups with coefficients
 $\mathbb{Z}$ are isomorphic to that of $S^{n-j-1}$ as ($n-j-1$)-dimensional manifolds.
\item In the family, there exists just $h_{n-1}-{\sum}_{j=0}^{n-2} h_j$ points.
\end{enumerate}

\end{proof}
Note that we cannot take $i_0$ as $1$ in the situation of Theorem \ref{thm:5} in general in this theorem. Moreover, for example, in the case $n=6$ or equivalently, $n-4=2$ holds, we can obtain an example we cannot obtain in the case where $i_0$ is $1$.

From Proposition \ref{prop:5}, related to Theorem \ref{thm:4} or \ref{thm:6}, we have the following.

\begin{Cor}
\label{cor:4}
In the situation of Theorem \ref{thm:4} or \ref{thm:6}, let the map $f$ be almost-special generic and the inequality $m-n>n-2$ hold. If for the source manifold ${M}^{\prime}$ of the resulting fold map $f^{\prime}$ obtained by a finite iteration of S-bubbling operations, the 1st homology group $H_1({M}^{\prime};\mathbb{Z})$ is trivial and $H_{n-2}(W_{{f}^{\prime}};\mathbb{Z})$ is not free, then the source manifold does not admit an almost special generic map into ${\mathbb{R}}^n$.  
\end{Cor}
  
Now we extend normal bubbling operations to obtain other maps and Reeb spaces. 

\begin{Def}
\label{def:5}
In Definition \ref{def:3}, let $S$ be the bouquet of finite connected and orientable closed submanifolds whose dimensions are smaller than $n$ of
 $P$ and $N(S)$, ${N(S)}_i$ and ${N(S)}_o$ be small regular neighborhoods
 of $S$ in $P$ such that ${N(S)}_i \subset N(S) \subset {N(S)}_o$
 holds and that these three are isotopic as regular neighborhoods. By a similar way,
 we define a similar operation and call the operation a {\it bubbling operation} to $f$. We call $Q_0:={\bar{f}}^{-1}(S) \bigcap q_f(Q)$, which is homeomorphic
 to $S$, the {\it generating polyhedron} of the bubbling operation.
\end{Def} 

We can define an {\it M-bubbling operation} and an {\it S-bubbling operation} similarly. 
We can also define a {\it trivial} bubbling operation similarly as in Example \ref{ex:1} (\ref{ex:1.1}).
FIGURE \ref{fig:8} and FIGURE \ref{fig:9} represent simple examples.

\begin{figure}
\includegraphics[width=40mm]{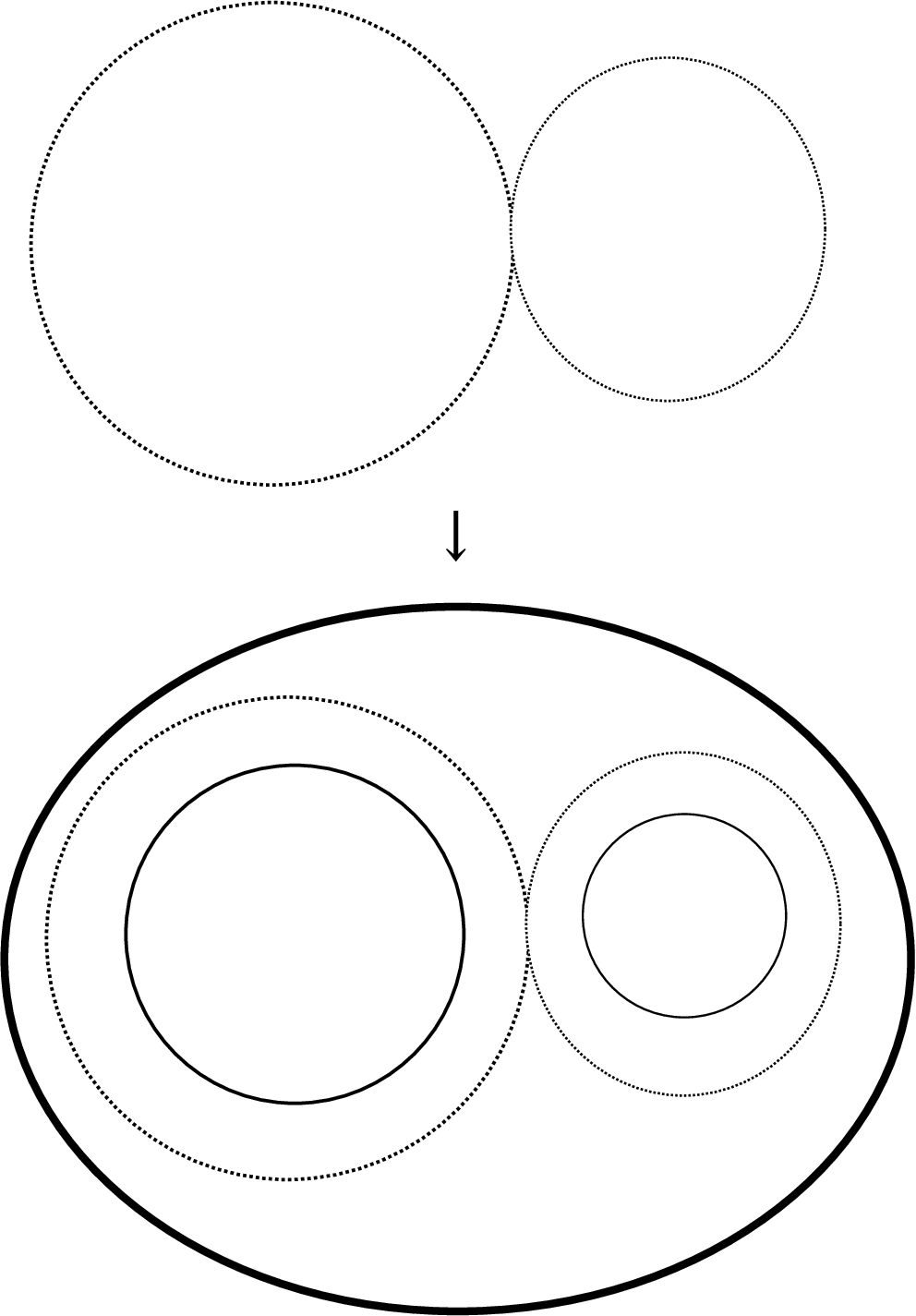}
\caption{The case where the generating polyhedron is a bouquet of two circles ($n=2$).}
\label{fig:8}
\end{figure}

\begin{figure}
\includegraphics[width=40mm]{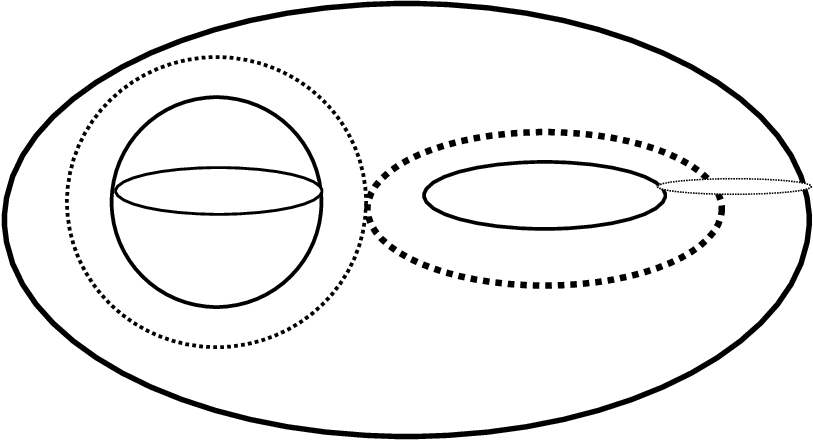}
\caption{The case where the generating polyhedron is a bouquet of a circle and a $2$-dimensional sphere ($n=3$).}
\label{fig:9}
\end{figure}

\begin{Prop}
\label{prop:6}
Let $f:M \rightarrow N$ be a fold map. Let $f^{\prime}:{M}^{\prime} \rightarrow N$ be a fold map obtained by an
 M-bubbling operation to $f$. Let $S$ be the generating polyhedron of the M-bubbling
 operation. Let $k$ be a positive integer and $S$ be represented as the bouquet of submanifolds $S_j$ where $j$ is an integer satisfying $1 \leq j \leq k$.  
Then, for any integer $0\leq i<n$, we have

$$H_{i}(W_{{f}^{\prime}};R) \cong H_{i}(W_f;R) \oplus {\oplus}_{j=1}^{k} (H_{i-(n-{\dim} S_j)}(S_j;R))$$

and we also have $H_{n}(W_{{f}^{\prime}};R) \cong H_{n}(W_f;R) \oplus R$.
\end{Prop}
\begin{proof}

Let $N(S)$ be a small regular neighborhood of $S$ in $W_f$ as in Definition \ref{def:4}. The
 space ${W_{{f}^{\prime}}}$ is obtained by attaching $q_{{f}^{\prime}}({q_f}^{-1}(N(S)))$ to $W_f$ on $N(S) \subset W_f$. We have the following
 exact sequence.

$\begin{CD}
@>   >> H_i(N(S);R) @>  >> H_i(W_f;R) \oplus H_i(q_{{f}^{\prime}}({q_f}^{-1}(N(S)));R)
\end{CD}$

$\begin{CD}
@>   >>  H_i(W_{f^{\prime}};R) @>   >> H_{i-1}(N(S);R) 
\end{CD}$

$\begin{CD}
@>   >> H_{i-1}(W_f;R) \oplus H_{i-1}(q_{{f}^{\prime}}({q_f}^{-1}(N(S)));R) @>   >>
\end{CD}$


$N(S)$ is
 regarded as
 a boundary connected sum of the manifolds $N(S_j)$ and each $N(S_j)$ is regarded as the
 total space of a bundle over $S_j$ whose fiber is an ($n-{\dim} S_j$)-dimensional standard closed disc. 

 $q_{{f}^{\prime}}({q_f}^{-1}(N(S)))$ is regarded as a connected sum of $n$-dimensional
 orientable closed and connected manifolds $Q_j$ satisfying the following two where $j$ is
 an integer satisfying $1 \leq j \leq k$.

\begin{itemize}
\item $Q_j$ is regarded as the total space of
 a linear $S^{n-{\dim} S_j}$-bundle over $S_j$.
\item The bundle $N(S_j)$ over $S_j$ is a subbundle of the bundle $Q_j$ whose fiber
 is $D^{n-{\dim} S_j} \subset S^{n-{\dim} S_j}$.
\end{itemize}

  $H_i(Q_j;R)$ is of the form ${\oplus}_{i^{\prime}=0}^{i} H_{i^{\prime}}(N(S_j);R) \otimes H_{i-i^{\prime}}(S^{n-{\dim} S_j};R)$. 
 The homomorphism from $H_i(N(S_j);R)$ into $H_i(Q_j;R)$ or ${\oplus}_{i^{\prime}=0}^{i} H_{i^{\prime}}(N(S_j);R) \otimes H_{i-i^{\prime}}(S^{n-{\dim} S_j};R)$ induced from the natural inclusion is injective and of the form $x \mapsto x \otimes 1 \in H_{i}(N(S_j);R) \otimes H_{0}(S^{n-{\dim} S_j};R)$. 

 Since $q_{{f}^{\prime}}({q_f}^{-1}(N(S)))$ is regarded as a connected sum of manifolds $Q_j$, for
 any integer $0<i<n$, $H_i(q_{{f}^{\prime}}({q_f}^{-1}(N(S)));R)$ is isomorphic to ${\oplus}_{j=1}^k H_i(Q_j;R)$
 and $H_0(q_{{f}^{\prime}}({q_f}^{-1}(N(S)));R) \cong H_n(q_{{f}^{\prime}}({q_f}^{-1}(N(S)));R) \cong R$ holds.

 By virtue of these observations, for any integer $0<i<n$, we have

$$H_{i}(W_{{f}^{\prime}};R) \cong H_{i}(W_f;R) \oplus
 {\oplus}_{j=1}^{k} (H_{i-(n-{\dim} S_j)}(S_j;R))$$. We also have $H_{0}(W_{{f}^{\prime}};R) \cong H_{0}(W_f;R)$ and $H_{n}(W_{{f}^{\prime}};R) \cong H_{n}(W_f;R) \oplus R$. Note that the first
  result holds for $i=0$ since $H_{i-(n-{\dim} S_j)}(S_j;R)$ is zero for $i=0$.
\end{proof}

\begin{Cor}
In the situation of Proposition \ref{prop:5}, if $H_{i-(n-{\dim} S_{j})}(S_{j};R)$ is free for
 any $1 \leq j \leq k$ and $H_i(W_f;R)$ is free, then $H_i(W_{{f}^{\prime}};R)$ is also free.
\end{Cor}

By applying M-bubbling operations, we have results similar to ones mentioned in Theorem \ref{thm:1} and Theorem \ref{thm:2} for example. However, some
 statements are a bit different. For example, we have the following.

\begin{Thm}
\label{thm:7}
Let $R$ be a PID. For any integer $0 \leq j \leq n$, we define $G_j$ as a free and finitely generated
 module over $R$ so that $G_0$ is trivial and that $G_n$ is not zero. Then, by
 a finite iteration of normal M-bubbling {\rm (}S-bubbling{\rm )} operations to a map $f$, we obtain a fold map $f^{\prime}$ such
  that $H_j(W_{{f}^{\prime}};R)$ is isomorphic to $H_j(W_f;R) \oplus G_j$. 
\end{Thm}
\begin{proof}
As done in Theorem \ref{thm:2}, we show the statement in the case where $f$ is the natural projection of the standard sphere and we can show this in general cases similarly.

Set $g_j={\rm rank} \quad G_j$.
We can choose $g_n \geq 1$ disjoint polyhedra in an open disc in an appropriate connected component of the regular value set as in Example \ref{ex:1} (\ref{ex:1.1}) so that the following hold.
\begin{enumerate}
\item Each polyhedron is a bouquet of spheres whose dimensions are not smaller than $1$ and not larger than $n-1$ or a point.
\item The number of $k$-dimensional spheres used in all of these bouquets is $g_{n-k}$.  
\end{enumerate}
We can apply (trivial) M-bubbling (S-bubbling) operations
 such that the generating
 polyhedra are inverse images of connected components in the previous disjoint
 polyhedra by the maps from the Reeb spaces to ${\mathbb{R}}^n$ as in Definition \ref{def:2} one after
 another starting from the given map $f$. By Proposition \ref{prop:6}, we have a desired map.

 Furthermore, if we choose $g_n \geq 1$ disjoint polyhedra in an open disc in an appropriate connected
  component of the regular value set as in Example \ref{ex:1} (\ref{ex:1.1}) such that the following two hold and that we can demonstrate a finite iteration of M-bubbling (resp. S-bubbling) operations similarly, then we also have a desired map.
\begin{enumerate}
\item Each polyhedron is a bouquet of manifolds whose dimensions are not smaller than $1$ and not larger than $n-1$ or a point.
\item The number of $k$-dimensional manifolds used in all of these bouquets is $g_{n-k}$ and their homology types with coefficient rings $R$ are same as that of a $k$-dimensional sphere.  
\end{enumerate}
\end{proof}

Compare this theorem to Theorem \ref{thm:2} (\ref{thm:2.1}) and (\ref{thm:2.2}) for example.

Last, we have the following, which explicitly implies that there exists a family of resulting maps obtained by finite iterations of bubbling operations to a given map such that the homology groups of the resulting Reeb spaces are mutually isomorphic.

\begin{Thm}
\label{thm:8}
Let $N={\mathbb{R}}^n$. 
Let $R$ be a PID. For any integer $0 \leq j \leq n$ and a finitely generated module $G_j$ over $R$ satisfying that $G_0$ is trivial and that $G_n$ is not zero. By
 a finite iteration of M-bubbling operations to a map $f:M \rightarrow {\mathbb{R}}^n$, let us obtain a fold map $f^{\prime}$ such
  that $H_j(W_{{f}^{\prime}};R)$ is isomorphic to $H_j(W_f;R) \oplus G_j$ and let us denote a polyhedron PL homeomorphic to the generating polyhedron of the $k$-th M-bubbling operation
 by $P_k$. Let $\{Q_k\}$ be a family of a finite number of bouquets of finite numbers of smooth, closed, connected and orientable manifolds whose dimensions are smaller than $n$ such that each $Q_k$ can be realized as a bouquet of smooth closed submanifolds in ${\mathbb{R}}^n$ and
 that $\{P_k\}$ is a subsequence of $\{Q_k\}$. Then there exists a finite iteration of bubbling operations to $f$ satisfying the following and we obtain a fold map such that the homology group of the Reeb space is isomorphic to that of $W_{{f}^{\prime}}$ in the assumption. 
\begin{enumerate}
\item For the $k$-the operation, the generating polyhedron is PL homeomorphic to $Q_k$.
\item For the step corresponding to $P_k$ and the corresponding generating polyhedron is $P_k$, the corresponding bubbling operation is an M-bubbling operation.
\item For steps except the steps just before, the corresponding bubbling operations are not M-bubbling operations.
\end{enumerate}
Furthermore, at each step, we can take the connected component of $f(M)-f(S(f))$ including the corresponding generating polyhedron and the corresponding connected component of the inverse image
 arbitrary. Furthermore, at the step corresponding to $P_k$, the resulting two connected components of the corresponding inverse image of a corresponding regular value can be any pair such that the connected sum is the original
 connected component and at the step not corresponding to any polyhedron in $\{P_k\}$, the resulting connected component of the corresponding inverse image of a corresponding regular value can be any manifold obtained by attaching a handle whose index is smaller than $m-n$ to the original connected component.
\end{Thm}

We only explain essential discussions in the proof and it is not difficult to understand the proof.

\begin{proof}[Important points of the proof]
We are enough to take each generating polyhedron in an open ball in an arbitrary connected component of the regular value set and at each step, perform a trivial bubbling operation after choosing an arbitrary connected component of the inverse image of the open ball: note that in this discussion, the assumptions $N={\mathbb{R}}^n$ etc. are essential. The results on resulting inverse images of corresponding regular values are based on the definition of (M-)bubbling operation.  
\end{proof}

\end{document}